\documentclass[12pt,a4paper]{article}
\usepackage{amsmath}
\usepackage{amsfonts}
\usepackage{amsthm}
\usepackage{amssymb}

\usepackage{enumerate}
\theoremstyle{plain}
\newtheorem{theo}{Theorem}[section]
\newtheorem{lem}[theo]{Lemma}
\newtheorem{prop}[theo]{Proposition}
\newtheorem{cor}[theo]{Corollary}%

\theoremstyle{definition}
\newtheorem{definition}[theo]{Definition}
\theoremstyle{remark}
\newtheorem{rem}[theo]{Remark}
\numberwithin{equation}{section}

\newcommand{\Z}{\mathbb{Z}}
\newcommand{\R}{\mathbb{R}}
\newcommand{\N}{\mathbb{N}}

\newcommand{\divrg}{\textrm{div}\,}

\title{Optimal three spheres inequality at the boundary for the Kirchhoff-Love plate's equation with Dirichlet conditions
\thanks{The first and the second authors are supported by FRA 2016 ``Problemi Inversi, dalla stabilit\`a alla ricostruzione'', Universit\`a degli Studi di Trieste. The second and the third authors are supported by Progetto GNAMPA 2017 ``Analisi di problemi inversi: stabilit\`a e ricostruzione'', Istituto Nazionale di Alta Matematica (INdAM).}}
\author{Giovanni Alessandrini\thanks{Dipartimento di Matematica e Geoscienze,
Universit\`a degli Studi di Trieste, via Valerio 12/1, 34127
Trieste, Italy. E-mail: \textsf{alessang@units.it}}, \  Edi
Rosset\thanks{Dipartimento di Matematica e Geoscienze,
Universit\`a degli Studi di Trieste, via Valerio 12/1, 34127
Trieste, Italy. E-mail: \textsf{rossedi@units.it}} \ and
Sergio Vessella\thanks{Dipartimento di Matematica e Informatica ``Ulisse Dini'', Universit\`a degli Studi di Firenze,Viale Morgagni 67/a,
50134 Firenze, Italy. E-mail:
\textsf{sergio.vessella@unifi.it}}}

\date{}

\begin{document}

\maketitle

\noindent \textbf{Abstract.} We prove a three spheres inequality with optimal exponent at the boundary for solutions to the Kirchhoff-Love plate's equation satisfying homogeneous Dirichlet conditions. This result implies the Strong Unique Continuation Property at the Boundary (SUCPB). Our approach is based on the method of Carleman estimates, and involves the construction of an ad hoc conformal mapping preserving the structure of the operator and the employment of a suitable reflection of the solution with respect to the flattened boundary which ensures the needed regularity of the extended solution. To the authors' knowledge, this is the first (nontrivial) SUCPB result for fourth-order equations with bi-Laplacian principal part.
\medskip

\medskip


\medskip

\medskip


\section{Introduction} \label{sec:
introduction}
The main purpose of this paper is to prove a Strong Unique Continuation Property at the Boundary (SUCPB) for the Kirchhoff-Love plate's equation. In order to introduce the subject of SUCPB we give some basic, although coarse, notion.

Let $\mathcal{L}$ be an elliptic operator of order $2m$, $m\in \N$, and let $\Omega$ be an open domain in $\mathbb{R}^N$, $N\geq2$. We say that $\mathcal{L}$ enjoys a SUCPB with respect to the Dirichlet boundary conditions if the following property holds true:

\begin{equation}\label{formulaz-sucpb}
\begin{cases}
 \mathcal{L}u=0, \mbox{ in } \Omega, \\
\frac{\partial^ju}{\partial n^j}=0, \mbox{ on } \Gamma, \quad\mbox{ for } j=0, 1, \ldots, m-1,   \\
\int_{\Omega\cap B_r(P)}u^2=\mathcal{O}(r^k), \mbox{ as } r\rightarrow 0, \forall k\in \mathbb{N},
\end{cases}\Longrightarrow \quad u\equiv 0 \mbox{ in } \Omega,
\end{equation}
where $\Gamma$ is an open portion (in the induced topology) of $\partial\Omega$, $n$ is outer unit normal, $P\in\Gamma$ and $B_r(P)$ is the ball of center $P$ and radius $r$.
Similarly, we say that $\mathcal{L}$ enjoys a SUCPB with respect to the set of normal boundary operators $\mathcal{B}_j$, $j\in J$, $\mathcal{B}_j$ of order $j$, $J\subset \{0, 1, \ldots, 2m-1\}$, $\sharp J = m$,  \cite{l:Fo}, if the analogous of \eqref{formulaz-sucpb} holds when the Dirichlet boundary conditions are replaced by
\begin{equation}
   \label{BC-generale}
   \mathcal{B}_ju=0, \quad \hbox{on } \Gamma, \quad \hbox{for } j\in J.
\end{equation}
The SUCPB has been studied for the second order elliptic operators in the last two decades, both in the case of homogeneous Dirichlet, Neumann and Robin boundary conditions, \cite{l:AdEsKe}, \cite{l:AdEs}, \cite{l:ARRV}, \cite{l:ApEsWaZh}, \cite{l:BaGa}, \cite{l:BoWo}, \cite{l:KeWa}, \cite{l:KuNy}, \cite{l:Si}. Although the conjecture that the SUCPB holds true when $\partial\Omega$ is of Lipschitz class is not yet proved, the SUCPB and the related quantitative estimates are today well enough understood for second-order elliptic equations.

Starting from the paper \cite{l:AlBeRoVe}, the SUCPB turned out to be a crucial property to prove optimal stability estimates for inverse elliptic boundary value problems with unknown boundaries. Mostly for this reason the investigation about the SUCPB has been successfully extended to second order parabolic equations \cite{l:CaRoVe}, \cite{l:DcRoVe}, \cite{l:EsFe}, \cite{l:EsFeVe}, \cite{l:Ve1} and to wave equation with time independent coefficients \cite{l:SiVe}, \cite{l:Ve2}.  For completeness we recall (coarsely) the formulation  of inverse boundary value problems with unknown boundaries in the elliptic context.

Assume that $\Omega$ is a bounded domain, with connected boundary $\partial \Omega$ of $C^{1,\alpha}$ class, and that $\partial\Omega$ is disjoint union of an accessible portion $\Gamma^{(a)}$ and of an inaccessible portion $\Gamma^{(i)}$. Given a symmetric, elliptic, Lipschitz matrix valued $A$ and $\psi \not\equiv 0$ such that $$\psi(x)=0, \mbox{ on } \Gamma^{(i)},$$
let $u$ be the solution to
\begin{equation*}
\left\{\begin{array}{ll}
\mbox{div}\left(A\nabla u\right)=0, \quad \hbox{in } \Omega,\\
u=\psi, \quad \hbox{on } \partial\Omega.
\end{array}\right.
\end{equation*}
Assuming that one knows
\begin{equation*}
\label{flux}
A\nabla u\cdot\nu,\quad \mbox{on } \Sigma,
\end{equation*}
where $\Sigma$ is an open portion of $\Gamma^{(a)}$, the inverse problem under consideration consists in determining the unknown boundary $\Gamma^{(i)}$. The proof of the uniqueness of $\Gamma^{(i)}$ is quite simple and requires the weak unique continuation property of elliptic operators. On the contrary, the optimal continuous dependence of $\Gamma^{(i)}$ from the Cauchy data $u$, $A\nabla u\cdot\nu$ on $\Sigma$,
which is of logarithmic rate (see \cite{l:DcRo}), requires quantitative estimates of strong unique continuation at the interior and at the boundary, like the three spheres inequality, \cite{l:Ku}, \cite{l:La} and the doubling inequality, \cite{l:AdEs}, \cite{l:GaLi}.

Inverse problems with unknown boundaries have been studied in linear elasticity theory for elliptic systems \cite{l:mr03}, \cite{l:mr04}, \cite{l:mr09}, and for fourth-order elliptic equations \cite{l:mrv07}, \cite{l:mrv09}, \cite{l:mrv13}. It is clear enough that the unavailability of the SUCPB precludes proving optimal stability estimates for these inverse problems with unknown boundaries.

In spite of the fact that the strong unique continuation in the interior for fourth-order elliptic equation of the form
\begin{equation}\label{bilaplacian}
\Delta^2u+\sum_{|\alpha|\leq 3}c_{\alpha} D^{\alpha}u=0
\end{equation}
where $c_{\alpha}\in L^{\infty}(\Omega)$, is nowadays well understood, \cite{l:CoGr}, \cite{l:CoKo},  \cite{l:Ge}, \cite{l:LBo}, \cite{l:LiNaWa}, \cite{l:mrv07}, \cite{l:Sh}, to the authors knowledge, the SUCPB for equation like \eqref{bilaplacian} has not yet proved even for Dirichlet boundary conditions. In this regard it is worthwhile to emphasize that serious difficulties occur in performing Carleman method (the main method to prove the unique continuation property) for bi-Laplace operator \textit{near the boundaries}, we refer to \cite{l:LeRRob} for a thorough discussion and wide references on the topics.

In the present paper we begin to find results in this direction for the Kirchhoff-Love equation, describing thin isotropic elastic plates

\begin{equation}
    \label{eq:equazione_piastra-int}
    L(v) := {\rm div}\left ({\rm div} \left ( B(1-\nu)\nabla^2 v + B\nu \Delta v I_2 \right ) \right )=0, \qquad\hbox{in
    } \Omega\subset\mathbb{R}^2,
\end{equation}
where $v$ represents the transversal displacement, $B$ is the \emph{bending stiffness} and $\nu$ the \emph{Poisson's coefficient} (see \eqref{eq:3.stiffness}--\eqref{eq:3.E_nu} for the precise definitions).

Assuming $B,\nu\in C^4(\overline{\Omega})$ and $\Gamma$ of $C^{6, \alpha}$ class, we prove our main results: a three spheres inequality at the boundary with optimal exponent (see Theorem \ref{theo:40.teo} for the precise statement) and, as a byproduct, the following SUCPB result (see Corollary \ref{cor:SUCP})

\begin{equation}\label{formulaz-sucpb-piastra}
\begin{cases}
 Lv=0, \mbox{ in } \Omega, \\
v =\frac{\partial v}{\partial n}=0, \mbox{ on } \Gamma,   \\
\int_{\Omega\cap B_r(P)}v^2=\mathcal{O}(r^k), \mbox{ as } r\rightarrow 0, \forall k\in \mathbb{N},
\end{cases}\Longrightarrow \quad v\equiv 0 \mbox{ in } \Omega.
\end{equation}
In our proof, firstly we flatten the boundary $\Gamma$ by introducing a suitable conformal mapping (see Proposition \ref{prop:conf_map}), then we combine a reflection argument (briefly illustrated below) and the Carleman estimate
\begin{equation}
    \label{eq:24.4-intr}
	\sum_{k=0}^3 \tau^{6-2k}\int\rho^{2k+\epsilon-2-2\tau}|D^kU|^2dxdy\leq C
	\int\rho^{6-\epsilon-2\tau}(\Delta^2 U)^2dxdy,
\end{equation}
for every $\tau\geq \overline{\tau}$ and for every $U\in C^\infty_0(B_{\widetilde{R}_0}\setminus\{0\})$, where $0<\varepsilon<1$ is fixed and $\rho(x,y)\sim \sqrt{x^2+y^2}$ as $(x,y)\rightarrow (0,0)$, see \cite[Theorem 6.8]{l:mrv07} and here Proposition \ref{prop:Carleman} for the precise statement.

To enter a little more into details, let us outline the main steps of our proof.

a) Since equation \eqref{eq:equazione_piastra-int} can be rewritten in the form
\begin{equation}
    \label{eq:equazione_piastra_non_div-intr}
    \Delta^2 v= -2\frac{\nabla B}{B}\cdot \nabla\Delta v + q_2(v) \qquad\hbox{in
    } \Omega,
\end{equation}
where $q_2$ is a second order operator, the equation resulting after flattening $\Gamma$ by a conformal mapping preserves the same structure of \eqref{eq:equazione_piastra_non_div-intr} and, denoting by $u$ the solution in the new coordinates, we can write
\begin{equation}
    \label{eq:15.1a-intro}
\begin{cases}
 \Delta^2 u= a\cdot \nabla\Delta u + p_2(u), \qquad\hbox{in
    } B_1^+, \\
u(x,0)=u_y(x,0) =0, \quad \forall x\in (-1,1)
\end{cases}
\end{equation}
where $p_2$ is a second order operator.

b) We use the following reflection of $u$, \cite{l:Fa}, \cite{l:Jo}, \cite{l:Sa},
\begin{equation*}
    \overline{u}(x,y)=\left\{
  \begin{array}{cc}
    u(x,y), & \hbox{ in } B_1^+\\
    w(x,y)=-[u(x,-y)+2yu_y(x,-y)+y^2\Delta u(x,-y)], & \hbox{ in } B_1^-
  \end{array}
\right.
\end{equation*}
which has the advantage of ensuring that $\overline{u}\in H^4(B_1)$ if $u\in H^4(B_1^+)$ (see Proposition \ref{prop:16.1}), and then we apply the Carleman estimate \eqref{eq:24.4-intr} to $\xi \overline{u}$, where $\xi$ is a cut-off function. Nevertheless we have still a problem. Namely

c) Derivatives of $u$ up to the sixth order occur in the terms on the right-hand side of the Carleman estimate involving negative value of $y$, hence such terms cannot be absorbed in a standard way by the left hand side. In order to overcome this obstruction, we use Hardy inequality, \cite{l:HLP34}, \cite{l:T67}, stated in Proposition \ref{prop:Hardy}.

The paper is organized as follows. In Section \ref{sec:
notation} we introduce some notation and definitions and state our main results, Theorem \ref{theo:40.teo} and Corollary \ref{cor:SUCP}. In Section \ref{sec:
flat_boundary} we state Proposition \ref{prop:conf_map}, which introduces the conformal map which
realizes a local flattening of the boundary which preserves the structure of the differential operator.
Section \ref{sec:
Preliminary} contains some auxiliary results which shall be used in the proof of the three spheres inequality in the case of flat boundaries, precisely Propositions \ref{prop:16.1} and \ref{prop:19.2} concerning the reflection w.r.t. flat boundaries and its properties, a Hardy's inequality (Proposition \ref{prop:Hardy}), the Carleman estimate for bi-Laplace operator (Proposition \ref{prop:Carleman}), and some interpolation estimates (Lemmas \ref{lem:Agmon} and \ref{lem:intermezzo}).
In Section \ref{sec:3sfere} we establish the three spheres inequality with optimal exponents for
the case of flat boundaries, Proposition \ref{theo:40.prop3}, and then we derive the proof of our main result, Theorem \ref{theo:40.teo}.
Finally, in the Appendix, we give the proof of Proposition \ref{prop:conf_map} and of
the interpolation estimates contained in Lemma \ref{lem:intermezzo}.

\section{Notation} \label{sec:
notation}

We shall generally denote points in $\R^2$ by $x=(x_1,x_2)$ or $y=(y_1,y_2)$, except for
Sections \ref{sec:
Preliminary} and \ref{sec:3sfere} where we rename $x,y$ the coordinates in $\R^2$.

In places we will use equivalently the symbols $D$ and $\nabla$ to denote the gradient of a function. Also we use the multi-index notation.

We shall denote by $B_r(P)$ the disc in $\R^2$ of radius $r$ and
center $P$, by $B_r$ the disk of radius $r$ and
center $O$, by $B_r^+$, $B_r^-$ the hemidiscs in $\R^2$  of radius $r$ and
center $O$ contained in the halfplanes $\R^2_+= \{x_2>0\}$, $\R^2_-= \{x_2<0\}$ respectively, and by $R_{
a,b}$ the rectangle $(-a,a)\times(-b,b)$.

Given a matrix $A =(a_{ij})$, we shall denote by $|A|$ its Frobenius norm $|A|=\sqrt{\sum_{i,j}a_{ij}^2}$.

Along our proofs, we shall denote by $C$ a constant which may change {}from line to line.

\begin{definition}
  \label{def:reg_bordo} (${C}^{k,\alpha}$ regularity)
Let $\Omega$ be a bounded domain in ${\R}^{2}$. Given $k,\alpha$,
with $k\in\N$, $0<\alpha\leq 1$, we say that a portion $S$ of
$\partial \Omega$ is of \textit{class ${C}^{k,\alpha}$ with
constants $r_{0}$, $M_{0}>0$}, if, for any $P \in S$, there
exists a rigid transformation of coordinates under which we have
$P=0$ and
\begin{equation*}
  \Omega \cap R_{r_0,2M_0r_0}=\{x \in R_{r_0,2M_0r_0} \quad | \quad
x_{2}>g(x_1)
  \},
\end{equation*}
where $g$ is a ${C}^{k,\alpha}$ function on
$[-r_0,r_0]$
satisfying
\begin{equation*}
g(0)=g'(0)=0,
\end{equation*}
\begin{equation*}
\|g\|_{{C}^{k,\alpha}([-r_0,r_0])} \leq M_0r_0,
\end{equation*}
where
\begin{equation*}
\|g\|_{{C}^{k,\alpha}([-r_0,r_0])} = \sum_{i=0}^k  r_0^i\sup_{[-r_0,r_0]}|g^{(i)}|+r_0^{k+\alpha}|g|_{k,\alpha},
\end{equation*}

\begin{equation*}
|g|_{k,\alpha}= \sup_ {\overset{\scriptstyle t,s\in [-r_0,r_0]}{\scriptstyle
t\neq s}}\left\{\frac{|g^{(k)}(t) - g^{(k)}(s)|}{|t-s|^\alpha}\right\}.
\end{equation*}

\end{definition}

We shall consider an isotropic thin elastic plate $\Omega\times \left[-\frac{h}{2},\frac{h}{2}\right]$, having middle plane $\Omega$ and width $h$. Under the Kirchhoff-Love theory, the transversal displacement $v$ satisfies the following fourth-order partial differential equation

\begin{equation}
    \label{eq:equazione_piastra}
    L(v) := {\rm div}\left ({\rm div} \left ( B(1-\nu)\nabla^2 v + B\nu \Delta v I_2 \right ) \right )=0, \qquad\hbox{in
    } \Omega.
\end{equation}

Here the \emph{bending stiffness} $B$ is given by

\begin{equation}
  \label{eq:3.stiffness}
  B(x)=\frac{h^3}{12}\left(\frac{E(x)}{1-\nu^2(x)}\right),
\end{equation}
and the \emph{Young's modulus} $E$ and the \emph{Poisson's coefficient} $\nu$ can be written in terms of the Lam\'{e} moduli as follows
\begin{equation}
  \label{eq:3.E_nu}
  E(x)=\frac{\mu(x)(2\mu(x)+3\lambda(x))}{\mu(x)+\lambda(x)},\qquad\nu(x)=\frac{\lambda(x)}{2(\mu(x)+\lambda(x))}.
\end{equation}

We shall make the following strong convexity assumptions on the Lam\'{e} moduli
\begin{equation}
  \label{eq:3.Lame_convex}
  \mu(x)\geq \alpha_0>0,\qquad 2\mu(x)+3\lambda(x)\geq\gamma_0>0, \qquad \hbox{ in } \Omega,
\end{equation}
where $\alpha_0$, $\gamma_0$ are positive constants.

It is easy to see that equation \eqref{eq:equazione_piastra} can be rewritten in the form

\begin{equation}
    \label{eq:equazione_piastra_non_div}
    \Delta^2 v= \widetilde{a}\cdot \nabla\Delta v + \widetilde{q}_2(v) \qquad\hbox{in
    } \Omega,
\end{equation}
with
\begin{equation}
    \label{eq:vettore_a_tilde}
    \widetilde{a}=-2\frac{\nabla B}{B},
\end{equation}

\begin{equation}
    \label{eq:q_2}
    \widetilde{q}_2(v)=-\sum_{i,j=1}^2\frac{1}{B}\partial^2_{ij}(B(1-\nu)+\nu B\delta_{ij})\partial^2_{ij} v.
\end{equation}
Let
\begin{equation}
   \label{eq:Omega_r_0}
\Omega_{r_0} = \left\{ x\in R_{r_0,2M_0r_0}\ |\ x_2>g(x_1) \right\},
\end{equation}

\begin{equation}
   \label{eq:Gamma_r_0}
\Gamma_{r_0} = \left\{(x_1,g(x_1))\ |\ x_1\in (-r_0,r_0)\right\},
\end{equation}
with
\begin{equation*}
g(0)=g'(0)=0,
\end{equation*}

\begin{equation}
   \label{eq:regol_g}
\|g\|_{{C}^{6,\alpha}([-r_0,r_0])} \leq M_0r_0,
\end{equation}
for some $\alpha\in (0,1]$.
Let
$v\in H^2(\Omega_{r_0})$ satisfy
\begin{equation}
  \label{eq:equat_u_tilde}
  L(v)= 0, \quad \hbox{ in } \Omega_{r_0},
\end{equation}

\begin{equation}
  \label{eq:Diric_u_tilde}
  v =  \frac{\partial v}{\partial n}= 0, \quad \hbox{ on } \Gamma_{r_0},
\end{equation}
where $L$ is given by \eqref{eq:equazione_piastra} and $n$ denotes the outer unit normal.

Let us assume that the Lam\'{e} moduli $\lambda,\mu$ satisfies the strong convexity condition \eqref{eq:3.Lame_convex} and the following regularity assumptions

\begin{equation}
  \label{eq:C4Lame}
  \|\lambda\|_{C^4(\overline{\Omega}_{r_0})}, \|\mu\|_{C^4(\overline{\Omega}_{r_0})}\leq \Lambda_0.
\end{equation}

The regularity assumptions \eqref{eq:3.Lame_convex}, \eqref{eq:regol_g} and \eqref{eq:C4Lame} guarantee that $v\in H^6(\Omega_r)$, see for instance \cite{l:a65}.

\begin{theo} [{\bf Optimal three spheres inequality at the boundary}]
    \label{theo:40.teo}
		Under the above hypotheses, there exist $c<1$ only depending on $M_0$ and $\alpha$, $C>1$ only depending on $\alpha_0$, $\gamma_0$, $\Lambda_0$, $M_0$, $\alpha$, such that, for every $r_1<r_2<c r_0<r_0$,
\begin{equation}
    \label{eq:41.1}
\int_{B_{r_2}\cap \Omega_{r_0}}v^2\leq C\left(\frac{r_0}{r_2}\right)^C\left(\int_{B_{r_1}\cap \Omega_{r_0}}v^2\right)^\theta\left(\int_{B_{r_0}\cap \Omega_{r_0}}v^2\right)^{1-\theta},
\end{equation}	
where
\begin{equation}
    \label{eq:41.2}
\theta = \frac{\log\left(\frac{cr_0}{r_2}\right)}{\log\left(\frac{r_0}{r_1}\right)}.
\end{equation}	
		
\end{theo}

\begin{cor} [{\bf Quantitative strong unique continuation at the boundary}]
   \label{cor:SUCP}
Under the above hypotheses and assuming $\int_{B_{r_0 }\cap\Omega_{r_0}}v^2>0$,
\begin{equation}
    \label{eq:suc1}
\int_{B_{r_1 }\cap\Omega_{r_0}}v^2 \geq \left(\frac{r_1}{r_0}\right)^{\frac{\log A}{\log \frac{r_2}{cr_0}}}
\int_{B_{r_0 }\cap\Omega_{r_0}}v^2,
\end{equation}
where
\begin{equation}
    \label{eq:suc2}
A= \frac{1}{C}\left(\frac{r_2}{r_0}\right)^C\frac{\int_{B_{r_2 }\cap\Omega_{r_0}}v^2}{\int_{B_{r_0 }\cap\Omega_{r_0}}v^2}<1,
\end{equation}
$c<1$ and $C>1$ being the constants appearing in Theorem \ref{theo:40.teo}.
\end{cor}

\begin{proof}
Reassembling the terms in \eqref{eq:41.1}, it is straightforward to obtain \eqref{eq:suc1}-\eqref{eq:suc2}. The SUCBP follows immediately.
\end{proof}

\section{Reduction to a flat boundary} \label{sec:
flat_boundary}

The following Proposition introduces a conformal map which flattens the boundary
$\Gamma_{r_0}$ and preserves the structure of equation \eqref{eq:equazione_piastra_non_div}.

\begin{prop} [{\bf Conformal mapping}]
    \label{prop:conf_map}
Under the hypotheses of Theorem \ref{theo:40.teo}, there exists an injective sense preserving differentiable map
\begin{equation*}
\Phi=(\varphi,\psi):[-1,1]\times[0,1]\rightarrow \overline{\Omega}_{r_0}
\end{equation*}
which is conformal, and it satisfies
\begin{equation}
   \label{eq:9.assente}
  \Phi((-1,1)\times(0,1))\supset B_{\frac{r_0}{K}}(0)\cap \Omega_{r_0},
\end{equation}
\begin{equation}
   \label{eq:9.2b}
  \Phi(([-1,1]\times\{0\})= \left\{ (x_1,g(x_1))\ |\ x_1\in [-r_1,r_1]\right\},
\end{equation}

\begin{equation}
   \label{eq:9.2a}
  \Phi(0,0)= (0,0),
\end{equation}

\begin{equation}
    \label{eq:gradPhi}
  \frac{c_0r_0}{2C_0}\leq |D\Phi(y)|\leq \frac{r_0}{2}, \quad \forall y\in [-1,1]\times[0,1],
\end{equation}
\begin{equation}
   \label{eq:gradPhiInv}
\frac{4}{r_0}\leq |D\Phi^{-1}(x)|\leq \frac{4C_0}{c_0r_0}, \quad\forall x\in \Phi([-1,1]\times[0,1]),
\end{equation}
\begin{equation}
   \label{eq:stimaPhi}
|\Phi(y)|\leq \frac{r_0}{2}|y|, \quad \forall y\in [-1,1]\times[0,1],
\end{equation}
\begin{equation}
   \label{eq:stimaPhiInv}
|\Phi^{-1}(x)| \leq
\frac{K}{r_0}|x|, \quad \forall x\in \Phi([-1,1]\times[0,1]),
\end{equation}
with
$K>8$, $0<c_0<C_0$ being constants only depending on $M_0$ and $\alpha$.

Letting
\begin{equation}
  \label{eq:def_sol_composta}
  u(y) = v(\Phi(y)), \quad y\in [-1,1]\times[0,1],
\end{equation}
then $u\in H^6((-1,1)\times(0,1))$ and it satisfies
\begin{equation}
    \label{eq:equazione_sol_composta}
    \Delta^2 u= a\cdot \nabla\Delta u + q_2(u), \qquad\hbox{in
    } (-1,1)\times(0,1),
\end{equation}
\begin{equation}
    \label{eq:Dirichlet_sol_composta}
    u(y_1,0)= u_{y_2}(y_1,0) =0, \quad \forall y_1\in (-1,1),
\end{equation}
where
\begin{equation*}
  a(y) = |\nabla \varphi(y)|^2\left([D\Phi(y)]^{-1}\widetilde{a}(\Phi(y))-2\nabla(|\nabla \varphi(y)|^{-2})\right),
\end{equation*}
$a\in C^3([-1,1]\times[0,1], \R^2)$, $q_2=\sum_{|\alpha|\leq 2}c_\alpha D^\alpha$ is a second order elliptic operator with coefficients $c_\alpha\in C^2([-1,1]\times[0,1])$,
 satisfying
\begin{equation}
    \label{eq:15.2}
    \|a\|_{ C^3([-1,1]\times[0,1], \R^2)}\leq M_1,\quad \|c_\alpha\|_{ C^2([-1,1]\times[0,1])}\leq M_1,
\end{equation}
with $M_1>0$ only depending on $M_0, \alpha, \alpha_0, \gamma_0, \Lambda_0$.
\end{prop}

The explicit construction of the conformal map $\Phi$ and the proof of the above Proposition are postponed to the Appendix.

\section{Preliminary results} \label{sec:
Preliminary}

In this paragraph, for simplicity of notation, we find it convenient to rename $x,y$ the coordinates in $\R^2$ instead of $y_1,y_2$.

Let $u\in H^6(B_1^+)$ be a solution to
\begin{equation}
    \label{eq:15.1a}
    \Delta^2 u= a\cdot \nabla\Delta u + q_2(u), \qquad\hbox{in
    } B_1^+,
\end{equation}
\begin{equation}
    \label{eq:15.1b}
    u(x,0)=u_y(x,0) =0, \quad \forall x\in (-1,1),
\end{equation}
with $q_2=\sum_{|\alpha|\leq 2}c_\alpha D^\alpha$,
\begin{equation}
    \label{eq:15.2_bis}
    \|a\|_{ C^3(\overline{B}_1^+, \R^2)}\leq M_1,\quad \|c_\alpha\|_{ C^2(\overline{B}_1^+)}\leq M_1,
\end{equation}
for some positive constant $M_1$.

Let us define the following extension of $u$ to $B_1$ (see \cite{l:Jo})
\begin{equation}
    \label{eq:16.1}
    \overline{u}(x,y)=\left\{
  \begin{array}{cc}
    u(x,y), & \hbox{ in } B_1^+\\
    w(x,y), & \hbox{ in } B_1^-
  \end{array}
\right.
\end{equation}
where
\begin{equation}
    \label{eq:16.2}
    w(x,y)= -[u(x,-y)+2yu_y(x,-y)+y^2\Delta u(x,-y)].
\end{equation}

\begin{prop}
    \label{prop:16.1}
Let
\begin{equation}
    \label{eq:16.3}
F:=a\cdot \nabla\Delta u + q_2(u).
\end{equation}
Then $F\in H^2(B_1^+)$, $\overline{u}\in H^4(B_1)$,
\begin{equation}
    \label{eq:16.4}
\Delta^2 \overline{u} = \overline{F},\quad \hbox{ in } B_1,
\end{equation}
where
\begin{equation}
    \label{eq:16.5}
    \overline{F}(x,y)=\left\{
  \begin{array}{cc}
    F(x,y), & \hbox{ in } B_1^+,\\
    F_1(x,y), & \hbox{ in } B_1^-,
  \end{array}
\right.
\end{equation}
and
\begin{equation}
    \label{eq:16.6}
    F_1(x,y)= -[5F(x,-y)-6yF_y(x,-y)+y^2\Delta F(x,-y)].
\end{equation}
\end{prop}

\begin{proof}
   Throughout this proof, we understand $(x,y)\in B_1^-$. It is easy to verify  that
\begin{equation}
    \label{eq:17.1}
    \Delta^2 w(x,y)= -[5F(x,-y)-6yF_y(x,-y)+y^2\Delta F(x,-y)]=F_1(x,y).
\end{equation}	
Moreover, by \eqref{eq:15.1b} and \eqref{eq:16.2},
\begin{equation}
    \label{eq:17.2}
    w(x,0)= -u(x,0) =0, \quad \forall x\in (-1,1).
\end{equation}	
By differentiating \eqref{eq:16.2} w.r.t. $y$, we have
\begin{equation}
    \label{eq:17.3bis}
    w_y(x,y)= -[u_y(x,-y)-2yu_{yy}(x,-y)+2y\Delta u(x,-y)-y^2(\Delta u_y)(x,-y)],
\end{equation}
so that, by \eqref{eq:15.1b},
\begin{equation}
    \label{eq:17.3}
    w_y(x,0)= -u_y(x,0) =0, \quad \forall x\in (-1,1).
\end{equation}
Moreover,
\begin{equation}
    \label{eq:17.6}
    \Delta w(x,y)= -[3 \Delta u(x,-y)-4u_{yy}(x,-y)-2y(\Delta u_y)(x,-y)+y^2(\Delta^2 u)(x,-y)],
\end{equation}
so that, recalling \eqref{eq:15.1b}, we have that, for every $x\in (-1,1)$,
\begin{multline}
    \label{eq:17.4}
    \Delta w(x,0)= -[3 \Delta u(x,0)-4u_{yy}(x,0)]= u_{yy}(x,0)
		= \Delta u (x,0).
\end{multline}
By differentiating \eqref{eq:17.6} w.r.t. $y$, we have
\begin{multline}
    \label{eq:18.1}
    (\Delta w_y)(x,y)= -[-5 (\Delta u_y)(x,-y)+4u_{yyy}(x,-y)+\\
		+
		2y(\Delta u_{yy})(x,-y)
		+2y(\Delta^2 u)(x,-y)-y^2(\Delta^2 u_y)(x,-y)],
\end{multline}
so that, taking into account \eqref{eq:15.1b}, it follows that, for every $x\in (-1,1)$,
\begin{multline}
    \label{eq:17.5}
    (\Delta w_y)(x,0)= -[-5 (\Delta u_y)(x,0)+4u_{yyy}(x,0)] =\\
		=-[-5 u_{yxx}(x,0)
		- u_{yyy}(x,0)] = u_{yyy}(x,0) = (\Delta u_y)(x,0).
\end{multline}
By \eqref{eq:17.2} and \eqref{eq:17.3}, we have that $\overline{u}\in H^2(B_1)$.
Let $\varphi\in C^\infty_0(B_1)$ be a test function. Then, integrating by parts and using \eqref{eq:17.1}, \eqref{eq:17.4}, \eqref{eq:17.5}, we have
\begin{multline}
    \label{eq:18.2}
    \int_{B_1}\Delta \overline{u} \Delta\varphi = \int_{B_1^+}\Delta u \Delta\varphi
		+\int_{B_1^-}\Delta w \Delta\varphi=\\
		=-\int_{-1 }^1 \Delta u(x,0)\varphi_y(x,0)+\int_{-1 }^1 (\Delta u_y)(x,0)\varphi(x,0)
		+\int_{B_1^+}(\Delta^2 u) \varphi +\\
		+\int_{-1 }^1 \Delta w(x,0)\varphi_y(x,0)-\int_{-1 }^1 (\Delta w_y)(x,0)\varphi(x,0)
		+\int_{B_1^-}(\Delta^2 w) \varphi=\\
		+\int_{B_1^+}F \varphi+\int_{B_1^-}F_1 \varphi
		=\int_{B_1}\overline{F} \varphi.
\end{multline}
Therefore
\begin{equation*}
    \int_{B_1}\Delta \overline{u} \Delta\varphi
		=\int_{B_1}\overline{F} \varphi, \quad \forall \varphi \in C^\infty_0(B_1),
\end{equation*}
so that \eqref{eq:16.4} holds and, by interior regularity esimates, $\overline{u}\in H^4(B_1)$.
\end{proof}

{}From now on, we shall denote by $P_k$, for $k\in \N$, $0\leq k\leq 3$, any differential operator of the form
\begin{equation*}
    \sum_{|\alpha|\leq k}c_\alpha(x)D^\alpha,
\end{equation*}
with $\|c_\alpha\|_{L^\infty}\leq cM_1$, where $c$ is an absolute constant.

\begin{prop}
    \label{prop:19.2}
		For every $(x,y)\in B_1^-$, we have
\begin{equation}
    \label{eq:19.1}
    F_1(x,y)= H(x,y)+(P_2(w))(x,y)+(P_3(u))(x,-y),
\end{equation}
where
\begin{multline}
    \label{eq:19.2}
    H(x,y)= 6\frac{a_1}{y}(w_{yx}(x,y)+u_{yx}(x,-y))+\\
		+6\frac{a_2}{y}(-w_{yy}(x,y)+u_{yy}(x,-y))
		-\frac{12a_2}{y}u_{xx}(x,-y),
\end{multline}
where $a_1,a_2$ are the components of the vector $a$.
Moreover, for every $x\in (-1,1)$,
\begin{equation}
    \label{eq:23.1}
    w_{yx}(x,0)+u_{yx}(x,0)=0,
\end{equation}
\begin{equation}
    \label{eq:23.2}
    -w_{yy}(x,0)+u_{yy}(x,0)=0,
\end{equation}
\begin{equation}
    \label{eq:23.3}
    u_{xx}(x,0)=0.
\end{equation}
\end{prop}
\begin{proof}
As before,  we understand $(x,y)\in B_1^-$.
Recalling \eqref{eq:16.2} and \eqref{eq:16.3}, it is easy to verify that
\begin{equation}
    \label{eq:19.3}
    F(x,-y)= (P_3(u))(x,-y),
\end{equation}
\begin{equation}
    \label{eq:20.1}
    -6yF_y(x,-y)= -6y(a\cdot \nabla \Delta u_y)(x,-y)+(P_3(u))(x,-y).
\end{equation}
Next, let us prove that
\begin{equation}
    \label{eq:20.2}
    y^2\Delta F(x,-y)= (P_2(w))(x,y)+(P_3(u))(x,-y).
\end{equation}
By denoting for simplicity $\partial_1 =\frac{\partial}{\partial x}$,
$\partial_2 =\frac{\partial}{\partial y}$, we have that
\begin{multline}
    \label{eq:20.3}
    y^2\Delta F(x,-y)= y^2(a_j\partial_j\Delta^2 u + 2\nabla a_j\cdot \nabla \partial_j\Delta u + \Delta a_j\partial_j\Delta u)(x,-y)+y^2\Delta(q_2(u))(x,-y)=\\
		=y^2(a_j\partial_j(a\cdot \nabla \Delta u+q_2 (u))(x,-y)+
		2y^2(\nabla a_j\cdot\nabla\partial_j \Delta u)(x,-y)+\\
		+y^2(\Delta q_2(u))(x,-y)+y^2(P_3(u))(x,-y)=\\
		=y^2(a_j a\cdot \nabla \Delta \partial_j u)(x,-y)+\\
		+2y^2(\nabla a_j\cdot\nabla\partial_j \Delta u)(x,-y)
		+y^2\Delta(q_2(u))(x,-y)+y^2(P_3(u))(x,-y).
\end{multline}
By \eqref{eq:16.2}, we have
\begin{equation*}
    y^2\Delta u(x,-y)=-w(x,y)-u(x,-y)-2yu_y(x,-y),
\end{equation*}
obtaining
\begin{multline}
    \label{eq:21.1}
		y^2(a_j a\cdot \nabla \partial_j\Delta  u)(x,-y)=
		(a_j a\cdot \nabla \partial_j(y^2\Delta  u))(x,-y)+
		(P_3(u))(x,-y)=\\
		=(P_2(w))(x,y)+(P_3(u))(x,-y).
\end{multline}
Similarly, we can compute
\begin{equation}
    \label{eq:21.2}
		2y^2(\nabla a_j \cdot \nabla \partial_j\Delta  u)(x,-y)=
		(P_2(w))(x,y)+(P_3(u))(x,-y),
\end{equation}
\begin{equation}
    \label{eq:21.3}
		y^2(\Delta  q_2(u))(x,-y)=
		(P_2(w))(x,y)+(P_3(u))(x,-y).
\end{equation}
Therefore, \eqref{eq:20.2} follows {}from \eqref{eq:20.3}--\eqref{eq:21.3}.

{}From \eqref{eq:16.6}, \eqref{eq:19.3}--\eqref{eq:20.2}, we have
\begin{equation}
    \label{eq:21.4}
		F_1(x,y)=6y(a\cdot \nabla\Delta u_y)(x,-y)
		+(P_2(w))(x,y)+(P_3(u))(x,-y).
\end{equation}
We have that
\begin{equation}
    \label{eq:21.5}
		6y(a\cdot \nabla\Delta u_y)(x,-y)=
		6y(a_1\Delta u_{xy})(x,-y)+6y(a_2\Delta u_{yy})(x,-y).
\end{equation}
By \eqref{eq:16.2}, we have
\begin{equation}
    \label{eq:22.1}
		w_{yx}(x,y)=-u_{yx}(x,-y)+2yu_{yyx}(x,-y)-2y(\Delta u_{x})(x,-y)
		+y^2(\Delta u_{yx})(x,-y),
\end{equation}
so  that
\begin{equation}
    \label{eq:22.2}
		y(\Delta u_{yx})(x,-y)=\frac{1}{y}(w_{yx}(x,y)+u_{yx}(x,-y))+(P_3(u))(x,-y).
\end{equation}
Again by \eqref{eq:16.2}, we have
\begin{multline}
    \label{eq:22.3}
		w_{yy}(x,y)=\\
		=3u_{yy}(x,-y)-2(\Delta u)(x,-y)-2y((u_{yyy})(x,-y)+2\Delta u_y(x,-y))
		-y^2(\Delta u_{yy})(x,-y)=\\
		=u_{yy}(x,-y)-2u_{xx}(x,-y)-y^2(\Delta u_{yy})(x,-y)+y(P_3(u))(x,-y),
\end{multline}
so that
\begin{equation}
    \label{eq:22.4}
		y(\Delta u_{yy})(x,-y)=\frac{1}{y}(-w_{yy}(x,y)+u_{yy}(x,-y)-2u_{xx}(x,-y))+(P_3(u))(x,-y).
\end{equation}
Therefore \eqref{eq:19.1}--\eqref{eq:19.2} follow by \eqref{eq:21.4}, \eqref{eq:21.5}, \eqref{eq:22.2} and \eqref{eq:22.4}.

The identity \eqref{eq:23.1} is an immediate consequence of \eqref{eq:22.1} and \eqref{eq:15.1b}.

By \eqref{eq:15.1b}, we have \eqref{eq:23.3} and
by \eqref{eq:22.3} and \eqref{eq:23.3},

\begin{equation*}
		-w_{yy}(x,0)+ u_{yy}(x,0) =2 u_{xx}(x,0) =0.
\end{equation*}

\end{proof}

For the proof of the three spheres inequality at the boundary we shall use the following Hardy's inequality (\cite[\S 7.3, p. 175]{l:HLP34}), for a proof see also \cite{l:T67}.

\begin{prop} [{\bf Hardy's inequality}]
    \label{prop:Hardy}
		Let $f$ be an absolutely continuous function defined in $[0,+\infty)$, such that
		$f(0)=0$. Then
\begin{equation}
    \label{eq:24.1}
		\int_1^{+\infty} \frac{f^2(t)}{t^2}dt\leq 4 \int_1^{+\infty} (f'(t))^2dt.
\end{equation}
\end{prop}

Another basic result we need to derive the three spheres inequality at the boundary is the following Carleman estimate, which was obtained in \cite[Theorem 6.8]{l:mrv07}.

\begin{prop} [{\bf Carleman estimate}]
    \label{prop:Carleman}
		Let $\epsilon\in(0,1)$. Let us define
\begin{equation}
    \label{eq:24.2}
		\rho(x,y) = \varphi\left(\sqrt{x^2+y^2}\right),
\end{equation}
where
\begin{equation}
    \label{eq:24.3}
		\varphi(s) = s\exp\left(-\int_0^s \frac{dt}{t^{1-\epsilon}(1+t^\epsilon)}\right).
\end{equation}
Then there exist $\overline{\tau}>1$, $C>1$, $\widetilde{R}_0\leq 1$, only depending on $\epsilon$, such that
\begin{equation}
    \label{eq:24.4}
	\sum_{k=0}^3 \tau^{6-2k}\int\rho^{2k+\epsilon-2-2\tau}|D^kU|^2dxdy\leq C
	\int\rho^{6-\epsilon-2\tau}(\Delta^2 U)^2dxdy,
\end{equation}
for every $\tau\geq \overline{\tau}$ and for every $U\in C^\infty_0(B_{\widetilde{R}_0}\setminus\{0\})$.
\end{prop}

\begin{rem}
   \label{rem:stima_rho}
	Let us notice that
	\begin{equation*}
e^{-\frac{1}{\epsilon}}s\leq \varphi(s)\leq s,
\end{equation*}
\begin{equation}
    \label{eq:stima_rho}
	e^{-\frac{1}{\epsilon}}\sqrt{x^2+y^2}\leq \rho(x,y)\leq \sqrt{x^2+y^2}.
\end{equation}	
	
\end{rem}

We shall need also the following interpolation estimates.

\begin{lem}
    \label{lem:Agmon}
Let $0<\epsilon\leq 1$ and $m\in \N$, $m\geq 2$. There exists an absolute constant
$C_{m,j}$ such that for every $v\in H^m(B_r^+)$,
\begin{equation}
    \label{eq:3a.2}
	r^j\|D^jv\|_{L^2(B_r^+)}\leq C_{m,j}\left(\epsilon r^m\|D^mv\|_{L^2(B_r^+)}
	+\epsilon^{-\frac{j}{m-j}}\|v\|_{L^2(B_r^+)}\right).
\end{equation}	
\end{lem}

See for instance \cite[Theorem 3.3]{l:a65}.

\begin{lem}
    \label{lem:intermezzo}
		Let $u\in H^6(B_1^+)$ be a solution to \eqref{eq:15.1a}--\eqref{eq:15.1b}, with $a$ and $q_2$ satisfying
		\eqref{eq:15.2_bis}. For every $r$, $0<r<1$, we have
		\begin{equation}
    \label{eq:12a.2}
	\|D^hu\|_{L^2(B_{\frac{r}{2}}^+)}\leq \frac{C}{r^h}\|u\|_{L^2(B_r^+)}, \quad \forall
	h=1, ..., 6,
\end{equation}	
where $C$ is a constant only depending on $\alpha_0$, $\gamma_0$ and $\Lambda_0$.
\end{lem}	
The proof of the above result is postponed to the Appendix.
		
\section{Three spheres inequality at the boundary and proof of the main theorem} \label{sec:3sfere}

\begin{theo} [{\bf Optimal three spheres inequality at the boundary - flat boundary case}]
    \label{theo:40.prop3}
		Let $u\in H^6(B_1^+)$ be a solution to \eqref{eq:15.1a}--\eqref{eq:15.1b}, with $a$ and $q_2$ satisfying
		\eqref{eq:15.2_bis}. Then there exist $\gamma\in (0,1)$, only depending on $M_1$ and an absolute constant $C>0$ such that, for every $r<R<\frac{R_0}{2}<R_0<\gamma$,
\begin{equation}
    \label{eq:40.1}
R^{2\epsilon}\int_{B_R^+}u^2\leq C(M_1^2+1)\left(\frac{R_0/2}{R}\right)^C\left(\int_{B_r^+}u^2\right)^{\widetilde{\theta}}\left(\int_{B_{R_0}^+}u^2\right)^{1-\widetilde{\theta}},
\end{equation}	
where
\begin{equation}
    \label{eq:39.1}
\widetilde{\theta} = \frac{\log\left(\frac{R_0/2}{R}\right)}{\log\left(\frac{R_0/2}{r/4}\right)}.
\end{equation}	
		
\end{theo}
\begin{proof}
Let $\epsilon \in (0,1)$ be fixed, for instance $\epsilon=\frac{1}{2}$. However, it is convenient to maintain the parameter $\epsilon$ in the calculations. Along this proof, $C$ shall denote  a positive constant which may change {}from line to line.
Let $R_0\in (0,\widetilde{R}_0)$ to be chosen later, where $\widetilde{R}_0$ has been introduced in Proposition \ref{prop:Carleman}, and let
\begin{equation}
    \label{eq:25.1}
0<r<R<\frac{R_0}{2}.
\end{equation}
Let $\eta\in C^\infty_0((0,1))$ such that
\begin{equation}
    \label{eq:25.2}
0\leq \eta\leq 1,
\end{equation}
\begin{equation}
    \label{eq:25.3}
\eta=0, \quad \hbox{ in }\left(0,\frac{r}{4}\right)\cup \left(\frac{2}{3}R_0,1\right),
\end{equation}
\begin{equation}
    \label{eq:25.4}
\eta=1, \quad \hbox{ in }\left[\frac{r}{2}, \frac{R_0}{2}\right],
\end{equation}
\begin{equation}
    \label{eq:25.6}
\left|\frac{d^k\eta}{dt^k}(t)\right|\leq C r^{-k}, \quad \hbox{ in }\left(\frac{r}{4}, \frac{r}{2}\right),\quad\hbox{ for } 0\leq k\leq 4,
\end{equation}
\begin{equation}
    \label{eq:25.7}
\left|\frac{d^k\eta}{dt^k}(t)\right|\leq C R_0^{-k}, \quad \hbox{ in }\left(\frac{R_0}{2}, \frac{2}{3}R_0\right),\quad\hbox{ for } 0\leq k\leq 4.
\end{equation}
Let us define
\begin{equation}
    \label{eq:25.5}
\xi(x,y)=\eta(\sqrt{x^2+y^2}).
\end{equation}
By a density argument, we may apply the Carleman estimate \eqref{eq:24.4} to $U=\xi \overline{u}$, where $\overline{u}$ has been defined in \eqref{eq:16.1}, obtaining
\begin{multline}
    \label{eq:26.1}
	\sum_{k=0}^3 \tau^{6-2k}\int_{B_{R_0}^+}\rho^{2k+\epsilon-2-2\tau}|D^k(\xi u)|^2
	+\sum_{k=0}^3 \tau^{6-2k}\int_{B_{R_0}^-}\rho^{2k+\epsilon-2-2\tau}|D^k(\xi w)|^2\leq \\
	\leq C
	\int_{B_{R_0}^+}\rho^{6-\epsilon-2\tau}|\Delta^2(\xi u)|^2+
	C\int_{B_{R_0}^-}\rho^{6-\epsilon-2\tau}|\Delta^2(\xi w)|^2,
\end{multline}
for $\tau\geq \overline{\tau}$ and $C$ an absolute constant.

By \eqref{eq:25.2}--\eqref{eq:25.5} we have

\begin{multline}
    \label{eq:26.2}
	|\Delta^2(\xi u)|\leq \xi|\Delta^2 u|+C\chi_{B_{r/2}^+\setminus B_{r/4}^+}
	\sum_{k=0}^3 r^{k-4}|D^k u|+ C\chi_{B_{2R_0/3}^+\setminus B_{R_0/2}^+}\sum_{k=0}^3 R_0^{k-4}|D^k u|,
\end{multline}
\begin{multline}
    \label{eq:26.3}
	|\Delta^2(\xi w)|\leq \xi|\Delta^2 w|+C\chi_{B_{r/2}^-\setminus B_{r/4}^-}
	\sum_{k=0}^3 r^{k-4}|D^k w|+ C\chi_{B_{2R_0/3}^-\setminus B_{R_0/2}^-}\sum_{k=0}^3 R_0^{k-4}|D^k w|.
\end{multline}
Let us set
\begin{multline}
    \label{eq:27.1}
	J_0 =\int_{B_{r/2}^+\setminus B_{r/4}^+}\rho^{6-\epsilon-2\tau}
	\sum_{k=0}^3 (r^{k-4}|D^k u|)^2+
	\int_{B_{r/2}^-\setminus B_{r/4}^-}\rho^{6-\epsilon-2\tau}
	\sum_{k=0}^3 (r^{k-4}|D^k w|)^2,
\end{multline}
\begin{multline}
    \label{eq:27.2}
	J_1 =\int_{B_{2R_0/3}^+\setminus B_{R_0/2}^+}\rho^{6-\epsilon-2\tau}
	\sum_{k=0}^3 (R_0^{k-4}|D^k u|)^2+
	\int_{B_{2R_0/3}^-\setminus B_{R_0/2}^-}\rho^{6-\epsilon-2\tau}
	\sum_{k=0}^3 (R_0^{k-4}|D^k w|)^2.
\end{multline}
By inserting \eqref{eq:26.2}, \eqref{eq:26.3} in \eqref{eq:26.1} we have
\begin{multline}
    \label{eq:27.3}
	\sum_{k=0}^3 \tau^{6-2k}\int_{B_{R_0}^+}\rho^{2k+\epsilon-2-2\tau}|D^k(\xi u)|^2
	+\sum_{k=0}^3 \tau^{6-2k}\int_{B_{R_0}^-}\rho^{2k+\epsilon-2-2\tau}|D^k(\xi w)|^2\leq \\
	\leq C
	\int_{B_{R_0}^+}\rho^{6-\epsilon-2\tau}\xi^2|\Delta^2 u|^2+
	C\int_{B_{R_0}^-}\rho^{6-\epsilon-2\tau}\xi^2|\Delta^2 w|^2+CJ_0+CJ_1,
\end{multline}
for $\tau\geq \overline{\tau}$, with $C$ an absolute constant.

By \eqref{eq:15.1a} and \eqref{eq:15.2_bis} we can estimate the first term in the right hand side of \eqref{eq:27.3} as follows
\begin{equation}
    \label{eq:28.1}
	\int_{B_{R_0}^+}\rho^{6-\epsilon-2\tau}\xi^2|\Delta^2 u|^2\leq
	CM_1^2\int_{B_{R_0}^+}\rho^{6-\epsilon-2\tau}\xi^2\sum_{k=0}^3|D^k u|^2.
\end{equation}

By \eqref{eq:17.1}, \eqref{eq:19.1} and by making the change of variables
$(x,y)\rightarrow(x,-y)$ in the integrals involving the function $u(x,-y)$,
we can estimate the second term in the right hand side of \eqref{eq:27.3} as follows
\begin{multline}
    \label{eq:28.2}
	\int_{B_{R_0}^-}\rho^{6-\epsilon-2\tau}\xi^2|\Delta^2 w|^2\leq
	C\int_{B_{R_0}^-}\rho^{6-\epsilon-2\tau}\xi^2|H(x,y)|^2+\\
	+CM_1^2\int_{B_{R_0}^-}\rho^{6-\epsilon-2\tau}\xi^2\sum_{k=0}^2|D^k w|^2+
	CM_1^2\int_{B_{R_0}^+}\rho^{6-\epsilon-2\tau}\xi^2\sum_{k=0}^3|D^k u|^2.
\end{multline}
Now, let us split the integral in the right hand side of \eqref{eq:28.1} and the second and third integrals in the right hand side of \eqref{eq:28.2} over the domains of integration $B_{r/2}^\pm\setminus B_{r/4}^\pm$, $B_{R_0/2}^\pm\setminus B_{r/2}^\pm$, $B_{2R_0/3}^\pm\setminus B_{R_0/2}^\pm$ and then let us insert \eqref{eq:28.1}--\eqref{eq:28.2} so rewritten in \eqref{eq:27.3}, obtaining
\begin{multline}
    \label{eq:28.4}
	\sum_{k=0}^3 \tau^{6-2k}\int_{B_{R_0}^+}\rho^{2k+\epsilon-2-2\tau}|D^k(\xi u)|^2
	+\sum_{k=0}^3 \tau^{6-2k}\int_{B_{R_0}^-}\rho^{2k+\epsilon-2-2\tau}|D^k(\xi w)|^2\leq \\
	\leq
	C\int_{B_{R_0}^-}\rho^{6-\epsilon-2\tau}\xi^2|H(x,y)|^2
	+CM_1^2\int_{B_{R_0/2}^- \setminus B_{r/2}^-}\rho^{6-\epsilon-2\tau}\sum_{k=0}^2|D^k w|^2+\\+
	CM_1^2\int_{B_{R_0/2}^+ \setminus B_{r/2}^+}\rho^{6-\epsilon-2\tau}\sum_{k=0}^3|D^k u|^2
	+C(M_1^2+1)(J_0+J_1),
\end{multline}
for $\tau\geq \overline{\tau}$, with $C$ an absolute constant.
\end{proof}
Next, by estimating {}from below the integrals in the left hand side of this last inequality reducing their domain of integration to $B_{R_0/2}^\pm\setminus B_{r/2}^\pm$, where $\xi=1$, we have
\begin{multline}
    \label{eq:29.1}
	\sum_{k=0}^3 \int_{B_{R_0/2}^+ \setminus B_{r/2}^+}\tau^{6-2k}
	(1-CM_1^2\rho^{8-2\epsilon-2k})\rho^{2k+\epsilon-2-2\tau}|D^k u|^2+\\
  +\int_{B_{R_0/2}^- \setminus B_{r/2}^-}\rho^{4+\epsilon-2\tau}|D^3 w|^2
	+\sum_{k=0}^2\int_{B_{R_0/2}^- \setminus B_{r/2}^-}\tau^{6-2k}
	(1-CM_1^2\rho^{8-2\epsilon-2k})\rho^{2k+\epsilon-2-2\tau}|D^k w|^2
	\leq \\
	\leq
	C\int_{B_{R_0}^-}\rho^{6-\epsilon-2\tau}\xi^2|H(x,y)|^2
	+C(M_1^2+1)(J_0+J_1),
\end{multline}
for $\tau\geq \overline{\tau}$, with $C$ an absolute constant.
		
Recalling \eqref{eq:stima_rho}, we have that, for $k=0,1,2,3$ and for
$R_0\leq R_1:=\min\{\widetilde{R}_0,2(2CM_1^2)^{-\frac{1}{2(1-\epsilon)}}\}$,

\begin{equation}
    \label{eq:30.1}
	1-CM_1^2\rho^{8-2\epsilon-2k}\geq \frac{1}{2}, \quad \hbox{ in }B_{R_0/2}^\pm,
\end{equation}
so that, inserting \eqref{eq:30.1} in \eqref{eq:29.1}, we have
\begin{multline}
    \label{eq:30.3}
	\sum_{k=0}^3 \tau^{6-2k}\int_{B_{R_0/2}^+ \setminus B_{r/2}^+}
	\rho^{2k+\epsilon-2-2\tau}|D^k u|^2
	+\sum_{k=0}^3 \tau^{6-2k}\int_{B_{R_0/2}^- \setminus B_{r/2}^-}
	\rho^{2k+\epsilon-2-2\tau}|D^k w|^2
	\leq \\
	\leq
	C\int_{B_{R_0}^-}\rho^{6-\epsilon-2\tau}\xi^2|H(x,y)|^2
	+C(M_1^2+1)(J_0+J_1),
\end{multline}
for $\tau\geq \overline{\tau}$, with $C$ an absolute constant.

By \eqref{eq:19.2} and \eqref{eq:15.2_bis}, we have that
\begin{equation}
    \label{eq:30.4}
	\int_{B_{R_0}^-}\rho^{6-\epsilon-2\tau}\xi^2|H(x,y)|^2\leq CM_1^2(I_1+I_2+I_3),
\end{equation}
with
\begin{equation}
    \label{eq:31.0.1}
	I_1=\int_{-R_0}^{R_0}\left(\int_{-\infty}^0\left|y^{-1}(w_{yy}(x,y)-
	(u_{yy}(x,-y))\rho^\frac{6-\epsilon-2\tau}{2}\xi\right|^2dy\right)dx.
\end{equation}
\begin{equation}
    \label{eq:31.0.2}
	I_2=\int_{-R_0}^{R_0}\left(\int_{-\infty}^0\left|y^{-1}(w_{yx}(x,y)+
	(u_{yx}(x,-y))\rho^\frac{6-\epsilon-2\tau}{2}\xi\right|^2dy\right)dx.
\end{equation}
\begin{equation}
    \label{eq:31.0.4}
	I_3=\int_{-R_0}^{R_0}\left(\int_{-\infty}^0\left|y^{-1}
	u_{xx}(x,-y)\rho^\frac{6-\epsilon-2\tau}{2}\xi\right|^2dy\right)dx.
\end{equation}
Now, let us see that, for $j=1,2,3$,
\begin{multline}
    \label{eq:31.1}
	I_j\leq
	C\int_{B_{R_0}^-}\rho^{6-\epsilon-2\tau}\xi^2|D^3 w|^2
	+C\tau^2\int_{B_{R_0}^-}\rho^{4-\epsilon-2\tau}\xi^2|D^2 w|^2+\\
	+C\int_{B_{R_0}^+}\rho^{6-\epsilon-2\tau}\xi^2|D^3 u|^2
	+C\tau^2\int_{B_{R_0}^+}\rho^{4-\epsilon-2\tau}\xi^2|D^2 u|^2
	+C(J_0+J_1),
\end{multline}
for $\tau\geq \overline{\tau}$, with $C$ an absolute constant.

Let us verify \eqref{eq:31.1} for $j=1$, the other cases following by using similar arguments.

By \eqref{eq:23.2}, we can apply Hardy's inequality \eqref{eq:24.1}, obtaining
\begin{multline}
    \label{eq:32.2}
	\int_{-\infty}^0\left|y^{-1}(w_{yy}(x,y)-
	(u_{yy}(x,-y))\rho^{\frac{6-\epsilon-2\tau}{2}}\xi\right|^2dy\leq\\
	\leq 4\int_{-\infty}^0\left|\partial_y\left[(w_{yy}(x,y)-
	(u_{yy}(x,-y))\rho^{\frac{6-\epsilon-2\tau}{2}}\xi\right]\right|^2dy\leq\\
	\leq 16 \int_{-\infty}^0\left(|w_{yyy}(x,y)|^2 +|u_{yyy}(x,-y)|^2\right)\rho^{6-\epsilon-2\tau}\xi^2dy+\\
	16 \int_{-\infty}^0\left(|w_{yy}(x,y)|^2 +|u_{yy}(x,-y)|^2\right)\left|\partial_y\left(\rho^{\frac{6-\epsilon-2\tau}{2}}\xi\right)\right|^2dy.
\end{multline}
Noticing that
\begin{equation}
    \label{eq:32.1}
	|\rho_y|\leq\left|\frac{y}{\sqrt{x^2+y^2}}\varphi'(\sqrt{x^2+y^2})\right|\leq 1,
\end{equation}
we can compute
\begin{multline}
    \label{eq:32.3}
	\left|\partial_y\left(\rho^{\frac{6-\epsilon-2\tau}{2}}(x,y)\xi(x,y)\right)\right|^2\leq
	2|\xi_y|^2\rho^{6-\epsilon-2\tau}+2\left|\left(\frac{6-\epsilon-2\tau}{2}\right)\xi \rho_y\rho^{\frac{4-\epsilon-2\tau}{2}}\right|^2\leq\\
	\leq 2\xi_y^2\rho^{6-\epsilon-2\tau}+2\tau^2\rho^{4-\epsilon-2\tau}\xi^2,
\end{multline}
for $\tau\geq \widetilde{\tau}:= \max\{\overline{\tau},3\}$, with $C$ an absolute constant.

By inserting \eqref{eq:32.3} in \eqref{eq:32.2}, by integrating over $(-R_0,R_0)$ and
by making the change of variables
$(x,y)\rightarrow(x,-y)$ in the integrals involving the function $u(x,-y)$, we derive
\begin{multline}
    \label{eq:33.0}
	I_1\leq C\int_{B_{R_0}^-}\xi^2\rho^{6-\epsilon-2\tau}|w_{yyy}|^2+
	C\int_{B_{R_0}^+}\xi^2\rho^{6-\epsilon-2\tau}|u_{yyy}|^2+\\
	+C\int_{B_{R_0}^-}\xi_y^2\rho^{6-\epsilon-2\tau}|w_{yy}|^2
	+C\int_{B_{R_0}^+}\xi_y^2\rho^{6-\epsilon-2\tau}|u_{yy}|^2+\\
	+C\tau^2\int_{B_{R_0}^-}\xi^2\rho^{4-\epsilon-2\tau}|w_{yy}|^2
	+C\tau^2\int_{B_{R_0}^+}\xi^2\rho^{4-\epsilon-2\tau}|u_{yy}|^2.
\end{multline}
Recalling \eqref{eq:25.2}--\eqref{eq:25.5}, we find \eqref{eq:31.1} for $j=1$.

Next, by \eqref{eq:30.3}, \eqref{eq:30.4} and \eqref{eq:31.1}, we have
\begin{multline}
    \label{eq:33.1}
	\sum_{k=0}^3 \tau^{6-2k}\int_{B_{R_0/2}^+ \setminus B_{r/2}^+}
	\rho^{2k+\epsilon-2-2\tau}|D^k u|^2
	+\sum_{k=0}^3 \tau^{6-2k}\int_{B_{R_0/2}^- \setminus B_{r/2}^-}
	\rho^{2k+\epsilon-2-2\tau}|D^k w|^2
	\leq \\
	\leq
	CM_1^2\int_{B_{R_0}^+}\rho^{6-\epsilon-2\tau}\xi^2|D^3u|^2+
	CM_1^2\int_{B_{R_0}^-}\rho^{6-\epsilon-2\tau}\xi^2|D^3w|^2+\\
	+CM_1^2\tau^2\int_{B_{R_0}^+}\rho^{4-\epsilon-2\tau}\xi^2|D^2u|^2
	+CM_1^2\tau^2\int_{B_{R_0}^-}\rho^{4-\epsilon-2\tau}\xi^2|D^2w|^2
	+C(M_1^2+1)(J_0+J_1),
\end{multline}
for $\tau\geq \widetilde{\tau}$, with $C$ an absolute constant.

Now, let us split the first four integrals in the right hand side of \eqref{eq:33.1}  over the domains of integration $B_{r/2}^\pm\setminus B_{r/4}^\pm$, $B_{2R_0/3}^\pm\setminus B_{R_0/2}^\pm$ and $B_{R_0/2}^\pm\setminus B_{r/2}^\pm$ and move on the left hand side the integrals over $B_{R_0/2}^\pm\setminus B_{r/2}^\pm$. Recalling \eqref{eq:stima_rho}, we obtain
\begin{multline}
    \label{eq:34.1}
	\sum_{k=2}^3 \int_{B_{R_0/2}^+ \setminus B_{r/2}^+}
	\tau^{6-2k}(1-CM_1^2\rho^{2-2\epsilon})\rho^{2k+\epsilon-2-2\tau}|D^k u|^2+\\
		+\sum_{k=2}^3 \int_{B_{R_0/2}^- \setminus B_{r/2}^-}
	\tau^{6-2k}(1-CM_1^2\rho^{2-2\epsilon})\rho^{2k+\epsilon-2-2\tau}|D^k w|^2+\\
	+\sum_{k=0}^1 \tau^{6-2k}\int_{B_{R_0/2}^+ \setminus B_{r/2}^+}
	\rho^{2k+\epsilon-2-2\tau}|D^k u|^2
	+\sum_{k=0}^1 \tau^{6-2k}\int_{B_{R_0/2}^- \setminus B_{r/2}^-}
	\rho^{2k+\epsilon-2-2\tau}|D^k w|^2
	\leq \\
	\leq
	C(\tau^2M_1^2+1)(J_0+J_1),
\end{multline}
for $\tau\geq \widetilde{\tau}$, with $C$ an absolute constant.

Therefore, for $R_0\leq R_2=\min\{R_1,2(2CM_1^2)^{-\frac{1}{2(1-\epsilon)}}\}$,
it follows that
\begin{multline}
    \label{eq:35.1}
	\sum_{k=0}^3 \tau^{6-2k}\int_{B_{R_0/2}^+ \setminus B_{r/2}^+}
	\rho^{2k+\epsilon-2-2\tau}|D^k u|^2+
	\sum_{k=0}^3 \tau^{6-2k}\int_{B_{R_0/2}^- \setminus B_{r/2}^-}
	\rho^{2k+\epsilon-2-2\tau}|D^k w|^2\leq\\	
	\leq
	C(\tau^2M_1^2+1)(J_0+J_1),
\end{multline}
for $\tau\geq \widetilde{\tau}$, with $C$ an absolute constant.

Let us estimate $J_0$ and $J_1$. {}From \eqref{eq:27.1} and recalling \eqref{eq:stima_rho}, we have
\begin{multline}
    \label{eq:36.1}
	J_0\leq\left(\frac{r}{4}\right)^{6-\epsilon-2\tau}\left\{
	\int_{B^+_{r/2}}\sum_{k=0}^3(r^{k-4}|D^k u|)^2+
	\int_{B^-_{r/2}}\sum_{k=0}^3(r^{k-4}|D^k w|)^2
	\right\}.
\end{multline}
By \eqref{eq:16.2}, we have that, for $(x,y)\in B^-_{r/2}$ and $k=0,1,2,3$,
\begin{equation}
    \label{eq:36.1bis}
	|D^k w|\leq C\sum_{h=k}^{2+k}r^{h-k}|(D^h u)(x,-y)|.
\end{equation}
By \eqref{eq:36.1}--\eqref{eq:36.1bis},
by making the change of variables
$(x,y)\rightarrow(x,-y)$ in the integrals involving the function $u(x,-y)$ and by using Lemma \ref{lem:intermezzo}, we get
\begin{multline}
    \label{eq:36.2}
	J_0\leq C\left(\frac{r}{4}\right)^{6-\epsilon-2\tau}
	\sum_{k=0}^5 r^{2k-8}\int_{B^+_{r/2}}|D^k u|^2
	 \leq C\left(\frac{r}{4}\right)^{-2-\epsilon-2\tau}\int_{B_r^+}|u|^2,
\end{multline}
where $C$ is an absolute constant. Analogously, we obtain
\begin{equation}
    \label{eq:37.1}
	J_1
	 \leq C\left(\frac{R_0}{2}\right)^{-2-\epsilon-2\tau}\int_{B_{R_0}^+}|u|^2.
\end{equation}

Let $R$ such that $r<R<\frac{R_0}{2}$. By \eqref{eq:35.1}, \eqref{eq:36.2}, \eqref{eq:37.1}, it follows that
\begin{multline}
    \label{eq:37.1bis}
	\tau^{6}R^{\epsilon-2-2\tau}\int_{B_{R}^+ \setminus B_{r/2}^+}
	|u|^2
	\leq\sum_{k=0}^3\tau^{6-2k}\int_{B_{R_0/2}^+ \setminus B_{r/2}^+}
	\rho^{2k+\epsilon-2-2\tau}|D^ku|^2\leq\\
	\leq C\tau^2 (M_1^2+1)\left[\left(\frac{r}{4}\right)^{-2-\epsilon-2\tau}\int_{B_r^+}|u|^2+
	\left(\frac{R_0}{2}\right)^{-2-\epsilon-2\tau}\int_{B_{R_0}^+}|u|^2
	\right],
\end{multline}
for $\tau\geq \widetilde{\tau}$, with $C$ an absolute constant.
Since $\tau>1$, we may rewrite the above inequality as follows
\begin{multline}
    \label{eq:37.2}
	R^{2\epsilon}\int_{B_{R}^+ \setminus B_{r/2}^+}
	|u|^2\leq
	C(M_1^2+1)\left[\left(\frac{r/4}{R}\right)^{-2-\epsilon-2\tau}\int_{B_r^+}|u|^2+
	\left(\frac{R_0/2}{R}\right)^{-2-\epsilon-2\tau}\int_{B_{R_0}^+}|u|^2
	\right],
\end{multline}
for $\tau\geq \widetilde{\tau}$, with $C$ an absolute constant.
By adding $R^{2\epsilon}\int_{B_{r/2}^+}|u|^2$ to both members of \eqref{eq:37.2}, and setting, for $s>0$,
\begin{equation*}
	\sigma_s=\int_{B_{s}^+}|u|^2,
\end{equation*}
we obtain
\begin{equation}
    \label{eq:38.1}
	R^{2\epsilon}\sigma_R\leq
	C(M_1^2+1)
	\left[\left(\frac{r/4}{R}\right)^{-2-\epsilon-2\tau}\sigma_r+
	\left(\frac{R_0/2}{R}\right)^{-2-\epsilon-2\tau}\sigma_{R_0}
	\right],
\end{equation}
for $\tau\geq \widetilde{\tau}$, with $C$ an absolute constant.

Let $\tau^*$ be such that
\begin{equation}
    \label{eq:38.2}
	\left(\frac{r/4}{R}\right)^{-2-\epsilon-2\tau^*}\sigma_r=
	\left(\frac{R_0/2}{R}\right)^{-2-\epsilon-2\tau^*}\sigma_{R_0},
\end{equation}
that is
\begin{equation}
    \label{eq:38.3}
	2+\epsilon+2\tau^*=\frac{\log (\sigma_{R_0}/\sigma_r)}{\log \left(\frac{R_0/2}{r/4}\right)}.
\end{equation}
Let us distinguish two cases:

\begin{enumerate}[i)]

\item
$\tau^*\geq \widetilde{\tau}$,

\item

$\tau^*< \widetilde{\tau}$,

\end{enumerate}
and set
\begin{equation}
    \label{eq:39.1bis}
	\widetilde{\theta}=\frac{\log \left(\frac{R_0/2}{R}\right)}{\log \left(\frac{R_0/2}{r/4}\right)}.
\end{equation}
In case i), it is possible to choose $\tau = \tau^*$ in \eqref{eq:38.1}, obtaining,
by \eqref{eq:38.2}--\eqref{eq:39.1bis},
\begin{equation}
    \label{eq:39.2}
	R^{2\epsilon}\sigma_R\leq C(M_1^2+1)\sigma_r^{\widetilde{\theta}}\sigma_{R_0}^{1-\widetilde{\theta}}.
\end{equation}
In case ii), since $\tau^*< \widetilde{\tau}$, {}from \eqref{eq:38.3}, we have
\begin{equation*}
\frac{\log (\sigma_{R_0}/\sigma_r)}{\log \left(\frac{R_0/2}{r/4}\right)}<2+\epsilon+2\widetilde{\tau},
\end{equation*}
so that, multiplying both members by $\log \left(\frac{R_0/2}{R}\right)$, it follows that
\begin{equation*}
\widetilde{\theta}\log\left(\frac{\sigma_{R_0}}{\sigma_r}\right)<\log\left(\frac{R_0/2}{R}\right)
^{2+\epsilon+2\widetilde{\tau}},
\end{equation*}
and hence
\begin{equation}
    \label{eq:39.3}
	\sigma_{R_0}^{\widetilde{\theta}}\leq \left(\frac{R_0/2}{R}\right)^{2+\epsilon+2\widetilde{\tau}}\sigma_r^{\widetilde{\theta}}.
\end{equation}
Then is follows trivially that
\begin{equation}
    \label{eq:39.4}
	R^{2\epsilon}\sigma_R\leq R^{2\epsilon}\sigma_{R_0}\leq
	R^{2\epsilon}\left(\frac{R_0/2}{R}\right)^{2+\epsilon+2\widetilde{\tau}}\sigma_r^{\widetilde{\theta}}\sigma_{R_0}^{1-\widetilde{\theta}}.
\end{equation}
Finally, by \eqref{eq:39.2} and \eqref{eq:39.4}, we obtain
\eqref{eq:40.1}.

\begin{proof}[Proof of Theorem \ref{theo:40.teo}]
Let $r_1<r_2<\frac{r_0R_0}{2K}<r_0$, where $R_0$ is chosen such that $R_0<\gamma<1$, where $\gamma$ has been introduced in Theorem
\ref{theo:40.prop3} and $K>1$ is the constant introduced in Proposition
\ref{prop:conf_map}.
Let us define
\begin{equation*}
   r=\frac{2r_1}{r_0}, \qquad R= \frac{Kr_2}{r_0}.
\end{equation*}
Recalling that $K>8$, it follows immediately that $r<R<\frac{R_0}{2}$.
Therefore, we can apply \eqref{eq:40.1} with $\epsilon=\frac{1}{2}$ to $u=v\circ\Phi$, obtaining
\begin{equation}
    \label{eq:3sfere_u}
\int_{B_R^+}u^2\leq \frac{C}{R^C}\left(\int_{B_r^+}u^2\right)^{\widetilde{\theta}}\left(\int_{B_{R_0}^+}u^2\right)^{1-\widetilde{\theta}},
\end{equation}
with
\begin{equation*}
\widetilde{\theta} = \frac{\log\left(\frac{R_0r_0}{2Kr_2}\right)}{\log\left(\frac{R_0r_0}{r_1}\right)}.
\end{equation*}
and $C>1$ only depending on $M_0$, $\alpha$, $\alpha_0$ e $\gamma_0$ and $\Lambda_0$.

{}From \eqref{eq:gradPhiInv}, \eqref{eq:stimaPhi}, \eqref{eq:stimaPhiInv}
and noticing that
\begin{equation*}
\widetilde{\theta} \geq \theta:=\frac{\log\left(\frac{R_0r_0}{2Kr_2}\right)}{\log\left(\frac{r_0}{r_1}\right)},
\end{equation*}
we obtain \eqref{eq:41.1}--\eqref{eq:41.2}.
\end{proof}

\section{Appendix} \label{sec:
Appendix}
\begin{proof}[Proof of Proposition \ref{prop:conf_map}]
Let us construct a suitable extension of $g$ to $[-2r_0,2r_0]$.
Let $P_6^\pm$ be the Taylor polynomial of order 6 and center $\pm r_0$
\begin{equation*}
	P_6^\pm(x_1)=\sum_{j=0}^6 \frac{g^{(j)}(\pm r_0)}{j!}(x_1-(\pm r_0))^j,
\end{equation*}
and let $\chi\in C^\infty_0(\R)$ be a function satisfying
\begin{equation*}
	0\leq\chi\leq 1,
\end{equation*}
\begin{equation*}
	\chi=1, \hbox{ for } |x_1|\leq r_0,
\end{equation*}
\begin{equation*}
	\chi=0, \hbox{ for } \frac{3}{2}r_0\leq |x_1|\leq 2r_0,
\end{equation*}
\begin{equation*}
	|\chi^{(j)}(x_1)|\leq \frac{C}{r_0^j}, \hbox{ for } r_0\leq |x_1|\leq \frac{3}{2}r_0, \forall j\in \N.
\end{equation*}
Let us define
\begin{equation*}
    \widetilde{g}=\left\{
  \begin{array}{cc}
    g, & \hbox{ for } x_1\in [-r_0,r_0],\\
    \chi P_6^+, & \hbox{ for } x_1\in [r_0, 2r_0],\\
		\chi P_6^-, & \hbox{ for } x_1\in [-2r_0, -r_0].
  \end{array}
\right.
\end{equation*}
It is a straightforward computation to verify that
\begin{equation}
   \label{eq:3.2}
	\widetilde g(x_1)=0, \hbox{ for } \frac{3}{2}r_0\leq |x_1|\leq 2r_0,
\end{equation}
\begin{equation}
   \label{eq:3.2bis}
	|\widetilde g(x_1)|\leq 2M_0r_0, \hbox{ for } |x_1|\leq 2r_0,
\end{equation}
so that the graph of $\widetilde g$ is contained in $R_{2r_0,2M_0r_0}$ and
\begin{equation}
    \label{eq:3.3}
	\|\widetilde g \|_{C^{6,\alpha}([-2r_0,2r_0])}\leq CM_0r_0,
\end{equation}
where $C$ is an absolute constant.
Let
\begin{equation}
   \label{eq:Omega_r_0_tilde}
\widetilde{\Omega}_{r_0} = \left\{ x\in R_{2r_0,2M_0r_0}\ |\ x_2>\widetilde{g}(x_1)\right\},
\end{equation}
and let $k\in H^1(\widetilde{\Omega}_{r_0} )$ be the solution to
\begin{equation}
  \label{eq:3.4}
  \left\{ \begin{array}{ll}
  \Delta k =0, &
  \hbox{in } \widetilde{\Omega}_{r_0},\\
   &  \\
  k_{x_1}(2r_0,x_2) =k_{x_1}(-2r_0,x_2)=0,     & \hbox{for } 0\leq x_2\leq 2M_0r_0,\\
  &  \\
 k(x_1,2M_0r_0) =1, & \hbox{for } -2r_0\leq x_1\leq 2r_0,\\
 &  \\
 k(x_1,\widetilde{g}(x_1)) =0,  &\hbox{for } -2r_0\leq x_1\leq 2r_0.\\
  \end{array}\right.
\end{equation}
Let us notice that $k\in C^{6,\alpha}\left(\overline{\widetilde{\Omega}}_{r_0} \right)$.
 Indeed, this regularity is standard away {}from any neighborhoods of the four points $(\pm2r_0,0)$, $(\pm 2r_0, 2M_0r_0)$ and, by making a even reflection of $k$ w.r.t. the lines $x_1 = \pm 2r_0$ in a neighborhood in $\widetilde{\Omega}_{r_0}$ of each of these points, we can apply Schauder estimates and again obtain the stated regularity.

By the maximum principle, $\min_{\overline{\widetilde{\Omega}}_{r_0}}k =
\min_{\partial \widetilde{\Omega}_{r_0}}k$. In view of the boundary conditions, this minimum value cannot be achieved
in the closed segment $\{x_2=2M_0r_0, |x_1|\leq 2r_0\}$. It cannot be achieved in the
segments $\{\pm 2r_0\}\times (0,2M_0r_0)$ since the boundary conditions over these segment contradict Hopf Lemma (see \cite{l:GT}). Therefore the minimum is attained on the boundary portion
$\{(x_1, \widetilde{g}(x_1) \ | \ x_1\in [-2r_0,2r_0]\}$, so that
$\min_{\overline{\widetilde{\Omega}}_{r_0}}k = 0$.
Similarly, $\max_{\overline{\widetilde{\Omega}}_{r_0}}k = 1$ and, moreover, by the strong maximum and minimum principles,
$0<k(x_1,x_2)<1$, for every $(x_1,x_2)\in \widetilde{\Omega}_{r_0}$.

Denoting by $\mathcal R$ be the reflection around the line $x_1=2r_0$, let

\begin{equation*}
\Omega^*_{r_0}=\widetilde{\Omega}_{r_0}\cup \mathcal R(\widetilde{\Omega}_{r_0})\cup(\{2r_0\}\times (0,2M_0r_0)),
\end{equation*}
and let $k^*$ be the extension of $k$ to $\overline{\Omega}^*_{r_0}$ obtained by making an even reflection of $k$ around the line $x_1=2r_0$.

Next, let us extend $k^*$ by periodicity w.r.t. the $x_1$ variable to the unbounded strip
\begin{equation*}
	S_{r_0} = \cup_{l\in \Z} (\Omega^*_{r_0} + 8r_0le_1).
\end{equation*}

By Schauder estimates and by the periodicity of $k^*$, it follows that
\begin{equation}
    \label{eq:5.1}
	\|\nabla k^*\|_{L^\infty(S_{r_0})}\leq \frac{C_0}{r_0},
\end{equation}
with $C_0$ only depending on $M_0$ and $\alpha$.
Therefore there exists $\delta_0= \delta_0(M_0, \alpha)$, $0<\delta_0\leq \frac{1}{4}$, such that
\begin{equation}
    \label{eq:5.2}
	k^*(x_1,x_2)\geq \frac{1}{2} \quad \forall (x_1,x_2)\in \R\times[(1-\delta_0)2M_0r_0,2M_0r_0].
\end{equation}
Since $k^*>0$ in $S_{r_0}$, by applying Harnack inequality and Hopf Lemma (see \cite{l:GT}), we have
\begin{equation*}
	\frac{\partial k^*}{\partial x_2}\geq \frac{c_0}{r_0}, \quad \hbox{ on }
	\partial S_{r_0},
\end{equation*}
with $c_0$ only depending on $M_0$ and $\alpha$.
Therefore, the function $k^*$ satisfies
\begin{equation*}
  \left\{ \begin{array}{ll}
  \Delta \left(\frac{\partial k^*}{\partial x_2}\right) =0, &
  \hbox{in } S_{r_0},\\
   &  \\
  \frac{\partial k^*}{\partial x_2}\geq \frac{c_0}{r_0},     & \hbox{on } \partial S_{r_0}.\\
  \end{array}\right.
\end{equation*}
Moreover, $\frac{\partial k^*}{\partial x_2}$, being continuous and periodic w.r.t. the variable $x_1$, attains its minimum in $\overline{S}_{r_0}$. Since this minimum value cannot be attained in $S_{r_0}$, it follows that
\begin{equation}
   \label{eq:6.1}
	\frac{\partial k^*}{\partial x_2}\geq \frac{c_0}{r_0}, \quad \hbox{ in }
	\overline{S}_{r_0}.
\end{equation}
Now, let $h$ be an harmonic conjugate of $-k$ in $\widetilde{\Omega}_{r_0}$, that is
\begin{equation}
   \label{eq:6.2}
  \left\{ \begin{array}{ll}
  h_{x_1} = k_{x_2}, &\\
   &  \\
  h_{x_2} = -k_{x_1}. &\\
  \end{array}\right.
\end{equation}
The map $\Psi : = h+ik$ is a conformal map in $\widetilde{\Omega}_{r_0}$,
\begin{equation}
   \label{eq:DPsi}
  D\Psi =\left( \begin{array}{ll}
  k_{x_2} &-k_{x_1}\\
   &  \\
  k_{x_1} &k_{x_2}\\
  \end{array}\right)
\end{equation}
so that
$|D\Psi| = \sqrt 2|\nabla k|$ and, by \eqref{eq:5.1} and \eqref{eq:6.1},
\begin{equation}
   \label{eq:6.3}
 \sqrt 2\frac{c_0}{r_0}\leq |D\Psi|\leq \sqrt 2\frac{C_0}{r_0}, \quad \hbox{in } \widetilde{\Omega}_{r_0}.
\end{equation}
Let us analyze the behavior of $\Psi$ on the boundary of $\widetilde{\Omega}_{r_0}$
\begin{equation*}
   \partial{\widetilde{\Omega}_{r_0}} = \sigma_1\cup \sigma_2\cup \sigma_3\cup \sigma_4,
\end{equation*}
where
\begin{equation*}
   \sigma_1 = \{(x_1, \widetilde{g}(x_1)),\ | \ x_1\in [-2r_0,2r_0]\},\qquad
	\sigma_2 = \{(2r_0, x_2),\ | \ x_2\in [0,2M_0r_0]\},
\end{equation*}
\begin{equation*}
   \sigma_3 = \{(x_1,2M_0r_0),\ | \ x_1\in [-2r_0,2r_0]\}, \qquad
	\sigma_4 = \{(-2r_0, x_2),\ | \ x_2\in [0,2M_0r_0]\}.
\end{equation*}
On $\sigma_1$, we have
\begin{equation*}
   \Psi(x_1, \widetilde{g}(x_1))= h((x_1, \widetilde{g}(x_1))) +i0,
\end{equation*}
\begin{equation*}
   \frac{\partial}{\partial x_1}h(x_1, \widetilde{g}(x_1)=
	h_{x_1}(x_1, \widetilde{g}(x_1)+ h_{x_2}(x_1, \widetilde{g}(x_1)\widetilde{g}'(x_1)
	=-\sqrt{1+[\widetilde{g}'(x_1)]^2}(\nabla k\cdot n)>0,
\end{equation*}
where $n$ is the outer unit normal.
Therefore $\Psi$ is injective on $\sigma_1$ and $\Psi(\sigma_1)$ is an
interval $[a,b]$ contained in the line $\{y_2=0\}$, with
\begin{equation*}
   a=h(-2r_0, 0), \quad b=h(2r_0, 0).
\end{equation*}
On $\sigma_2$, we have
\begin{equation*}
   \Psi(2r_0, x_2)= h(2r_0, x_2)+ik(2r_0, x_2),
\end{equation*}
\begin{equation*}
   h_{x_2}(2r_0, x_2)=-k_{x_1}(2r_0, x_2)=0,
\end{equation*}
and similarly in $\sigma_4$,
so that $h(-2r_0, x_2)\equiv a$ and  $h(2r_0, x_2)\equiv b$ for $x_2\in[0,2M_0r_0]$ whereas, by \eqref{eq:6.1}, $k$ is increasing w.r.t. $x_2$. Therefore $\Psi$ is injective on $\sigma_2\cup \sigma_4$,
and maps $\sigma_2$ into the segment $\{b\}\times[0,1]$ and $\sigma_4$ into the segment $\{a\}\times[0,1]$.

On $\sigma_3$, we have
\begin{equation*}
   \Psi(x_1, 2M_0r_0)= h(x_1, 2M_0r_0) +i1,
\end{equation*}
\begin{equation*}
	h_{x_1}(x_1, 2M_0r_0) = k_{x_2}(x_1, 2M_0r_0)>0,
\end{equation*}
so that $h$ is increasing in $[-2r_0,2r_0]$, $\Psi$ is injective on $\sigma_3$ and $\Psi(\sigma_3)$ is the interval $[a,b]\times\{1\}$.

Therefore $\Psi$ maps in a bijective way the boundary of $\widetilde{\Omega}_{r_0}$  into the boundary of $[a,b]\times [0,1]$. Moreover, we have
\begin{equation}
   \label{eq:b-a}
   b-a= \int_{-2r_0}^{2r_0}h_{x_1}(x_1,2M_0r_0)dx_1 =
	 \int_{-2r_0}^{2r_0}k_{x_2}(x_1,2M_0r_0)dx_1.
\end{equation}
By \eqref{eq:5.1}, \eqref{eq:6.1} and \eqref{eq:b-a} the following estimate holds
\begin{equation}
   \label{eq:b-a_bis}
   4c_0\leq b-a\leq 4C_0.
\end{equation}	
By \eqref{eq:6.3}, we can apply the global inversion theorem, ensuring that
\begin{equation*}
	\Psi^{-1}: [a,b]\times [0,1]\rightarrow \overline{\widetilde{\Omega}}_{r_0}
\end{equation*}
is a conformal diffeomorphism. Moreover,
\begin{equation}
   \label{eq:DPsi_inversa}
  D(\Psi^{-1}) =\frac{1}{|\nabla k|^2}\left( \begin{array}{ll}
  k_{x_2} &k_{x_1}\\
   &  \\
  -k_{x_1} &k_{x_2}\\
  \end{array}\right),
\end{equation}
\begin{equation}
\label{eq:8.1}
\frac{\sqrt 2}{C_0}r_0\leq |D\Psi^{-1}|= \frac{\sqrt 2}{|\nabla k|}\leq \frac{\sqrt 2}{c_0}r_0, \quad \hbox{in } [a,b]\times [0,1].
\end{equation}
Now, let us see that the set $\Psi(\Omega_{r_0})$ contains a closed rectangle having one basis contained in the line $\{y_2=0\}$ and whose sides can be estimated in terms of $M_0$ and $\alpha$. To this aim we need to estimate the distance of $\Psi(0,0)=(\overline{\xi}_1,0)$ {}from the edges $(a,0)$ and $(b,0)$ of the rectangle $[a,b]\times[0,1]$. Recalling that $\widetilde{g}\equiv 0$ for
$\frac{3}{2}r_0\leq |x_1|\leq 2r_0$, we have that $\sigma_1$ contains the segments
$\left[-2r_0,-\frac{3}{2}r_0\right]\times \{0\}$, $\left[\frac{3}{2}r_0,2r_0\right]\times \{0\}$, so that
\begin{equation}
   \label{eq:segmentino}
   h(2r_0,0)-h\left(\frac{3}{2}r_0,0\right)= \int_{\frac{3}{2}r_0}^{2r_0}h_{x_1}(x_1,0)dx_1 =
	 \int_{\frac{3}{2}r_0}^{2r_0}k_{x_2}(x_1,0)dx_1.
\end{equation}
By \eqref{eq:5.1}, \eqref{eq:6.1} and \eqref{eq:segmentino} we derive
\begin{equation}
   \label{eq:segmentino_bis}
   \frac{c_0}{2}\leq h(2r_0,0)-h\left(\frac{3}{2}r_0,0\right)\leq \frac{C_0}{2}.
\end{equation}	
Similarly,
\begin{equation}
   \label{eq:segmentino_ter}
   \frac{c_0}{2}\leq h\left(-\frac{3}{2}r_0,0\right)-h(-2r_0,0)\leq \frac{C_0}{2}.
\end{equation}	
Since $h$ is injective and maps $\sigma_1$ into $[a,b]\times\{0\}$, it follows that
\begin{equation*}
   |\Psi(0,0)-(a,0)| = h(0,0)-h(-2r_0,0) \geq\frac{c_0}{2},
\end{equation*}	
\begin{equation*}
	|\Psi(0,0)-(b,0)| = h(2r_0,0) - h(0,0) \geq\frac{c_0}{2}.
\end{equation*}	
Possibly replacing $c_0$ with $\min\{c_0,2\}$, we obtain that
$\overline{B}^+_{\frac{c_0}{2}}(\Psi(O))\subset [a,b]\times [0,1]$.
By \eqref{eq:8.1},
\begin{equation*}
	|\Psi^{-1}(\xi)| = |\Psi^{-1}(\xi)-\Psi^{-1}(\Psi(O))|\leq\frac{\sqrt 2}{2}r_0<r_0, \qquad \forall \xi \in B^+_{\frac{c_0}{2}}(\Psi(O)),
\end{equation*}	
so that $\Psi^{-1}\left(B^+_{\frac{c_0}{2}}(\Psi(O))\right)\subset \Omega_{r_0}$,
\begin{equation*}
\Psi(\Omega_{r_0})\supset B^+_{\frac{c_0}{2}}(\Psi(O))\supset R,
\end{equation*}
where $R$ is the rectangle
\begin{equation*}
R= \left(\overline{\xi}_1-\frac{c_0}{2\sqrt 2}, \overline{\xi}_1+\frac{c_0}{2\sqrt 2}\right)\times \left(0,\frac{c_0}{2\sqrt 2}\right).
\end{equation*}
Let us consider the homothety
\begin{equation*}
\Theta:[a,b]\times [0,1] \rightarrow\R^2,
\end{equation*}
\begin{equation*}
\Theta(\xi_1,\xi_2) = \frac{2\sqrt 2}{c_0}(\xi_1-\overline{\xi}_1,\xi_2),
\end{equation*}
which satisfies
\begin{equation*}
\Theta(\Psi(O)) = O, \qquad D\Theta = \frac{2\sqrt 2}{c_0} I_2,
\end{equation*}
\begin{equation*}
\Theta([a,b]\times [0,1]) =R^*, \qquad R^* =\left[\frac{2\sqrt 2}{c_0}(a-\overline{\xi}_1),
\frac{2\sqrt 2}{c_0}(b-\overline{\xi}_1)\right]\times
\left[0,
\frac{2\sqrt 2}{c_0}\right],
\end{equation*}
\begin{equation*}
\Theta(\overline{R}) = [-1,1]\times [0,1],
\end{equation*}
\begin{equation*}
D(\Theta\circ \Psi)(x) = \frac{2\sqrt 2}{c_0}D\Psi(x).
\end{equation*}
Its inverse
\begin{equation*}
\Theta^{-1}:R^*\rightarrow [a,b]\times [0,1],
\end{equation*}
\begin{equation*}
\Theta^{-1}(y_1,y_2) = \frac{c_0}{2\sqrt 2}(y_1+\overline{\xi}_1,y_2),
\end{equation*}
satisfies
\begin{equation*}
D\Theta^{-1}= \frac{c_0} {2\sqrt 2}I_2,
\end{equation*}
\begin{equation*}
D((\Theta\circ \Psi)^{-1})(y) = \frac{c_0}{2\sqrt 2}D\Psi^{-1}(\Theta^{-1}(y)).
\end{equation*}
Let us define
\begin{equation*}
\Phi =(\Theta\circ \Psi)^{-1}).
\end{equation*}
We have that $\Phi$ is a conformal diffeomorphism {}from $R^*$ into $\widetilde{\Omega}_{r_0}$ such that
\begin{equation*}
\Omega_{r_0}\supset \Psi^{-1}(R)=\Phi((-1,1)\times(0,1)),
\end{equation*}
\begin{equation}
   \label{eq:gradPhibis}
\frac{c_0r_0}{2C_0}\leq |D\Phi(y)|\leq \frac{r_0}{2},
\end{equation}
\begin{equation}
   \label{eq:gradPhiInvbis}
\frac{4}{r_0}\leq |D\Phi^{-1}(x)|\leq \frac{4C_0}{c_0r_0}.
\end{equation}
By \eqref{eq:gradPhi}, we have that, for every $y\in [-1,1]\times [0,1]$,
\begin{equation}
   \label{eq:stimaPhibis}
|\Phi(y)|= |\Phi(y)-\Phi(O)|\leq \frac{r_0}{2}|y|.
\end{equation}
Given any $x(x_1,x_2)\in \overline{\Omega}_{r_0}$, let $x^* =(x_1,g(x_1))$.
We have
\begin{equation*}
|x-x^*| = |x_2 - g(x_1)| \leq|x_2|+ |g(x_1)-g(0)|\leq (M_0+1)|x|,
\end{equation*}
and, since the segment joining $x$ and $x^*$ is contained in $\overline{\Omega}_{r_0}$,
by \eqref{eq:gradPhiInvbis} we have
\begin{equation}
   \label{eq:stimaPhiInv1}
|\Phi^{-1}(x)-\Phi^{-1}(x^*)|\leq \frac{4C_0}{c_0r_0}(M_0+1)|x|.
\end{equation}
Let un consider the arc $\tau(t)= \Phi^{-1}(t,g_1(t))$, for $t\in [0,x_1]$.
Again by \eqref{eq:gradPhiInvbis}, we have
\begin{multline}
   \label{eq:stimaPhiInv2}
|\Phi^{-1}(x^*)| = |\Phi^{-1}(x^*)-\Phi^{-1}(O)|  =\tau(x_1) -\tau(0) \leq\\
\leq \left|\int_0^{x_1}\tau'(t)dt \right|\leq \frac{4C_0}{c_0r_0}\sqrt{M_0^2+1}\ |x|.
\end{multline}
By \eqref{eq:stimaPhiInv1}, \eqref{eq:stimaPhiInv2}, we have
\begin{equation}
   \label{eq:stimaPhiInvbis}
|\Phi^{-1}(x)| \leq
\frac{K}{r_0}|x|,
\end{equation}
with $K=\frac{4C_0}{c_0}(M_0+1+\sqrt{M_0^2+1})>8$.
{}From this last inequality, we have that
\begin{equation*}
\Phi^{-1}\left(\Omega_{r_0}\cap B_{\frac{r_0}{K}}\right)\subset B_1^+\subset (-1,1)\times(0,1), \qquad \Phi((-1,1)\times(0,1))\supset \Omega_{r_0}\cap B_{\frac{r_0}{K}}.
\end{equation*}

Let $\Phi = (\varphi, \psi)$.
We have that
\begin{equation}
   \label{eq:DPhi}
  D\Phi =\left( \begin{array}{ll}
  \varphi_{y_1} &\varphi_{y_2}\\
   &  \\
  -\varphi_{y_2} &\varphi_{y_1}\\
  \end{array}\right),
\end{equation}
\begin{equation}
   \label{eq:32.1bisluglio}
  det(D\Phi(y)) = |\nabla\varphi(y)|^2,
\end{equation}
\begin{equation}
   \label{eq:DPhi_inversa}
  (D\Phi)^{-1} =\frac{1}{|\nabla \varphi|^2}\left( \begin{array}{ll}
  \varphi_{y_1} &-\varphi_{y_2}\\
   &  \\
  \varphi_{y_2} &\varphi_{y_1}\\
  \end{array}\right).
\end{equation}
Concerning the function $u(y) = v(\Phi(y))$, we can compute
\begin{equation}
   \label{eq:32.3luglio}
   (\nabla v) (\Phi(y)) = [(D\Phi(y))^{-1}]^T\nabla u(y),
\end{equation}
\begin{equation}
   \label{eq:32.2luglio}
   (\Delta v) (\Phi(y)) = \frac{1}{|det(D\Phi(y)|}\divrg(A(y)\nabla u(y)),
\end{equation}
where
\begin{equation}
   \label{eq:33.0luglio}
   A(y) = |det(D\Phi(y)| (D\Phi(y))^{-1} [(D\Phi(y))^{-1}]^T.
\end{equation}
By \eqref{eq:DPhi}--\eqref{eq:DPhi_inversa}, we obtain that
\begin{equation}
   \label{eq:33.0bisluglio}
   A(y) = I_2,
\end{equation}
so that
\begin{equation}
   \label{eq:33.0terluglio}
   (\Delta v) (\Phi(y)) = \frac{1}{|\nabla \varphi(y)|^2}\Delta u(y),
\end{equation}
\begin{equation}
   \label{eq:33.1luglio}
   (\Delta^2 v) (\Phi(y)) = \frac{1}{|\nabla \varphi(y)|^2}\Delta \left(\frac{1}{|\nabla \varphi(y)|^2}\Delta u(y)\right).
\end{equation}
By using the above formulas, some computations allow to derive \eqref{eq:equazione_sol_composta}--\eqref{eq:15.2} {}from \eqref{eq:equazione_piastra_non_div}.

Finally, the boundary conditions \eqref{eq:Dirichlet_sol_composta} follow {}from \eqref{eq:32.3luglio}, \eqref{eq:9.2b} and \eqref{eq:Diric_u_tilde}.
\end{proof}

\begin{proof}[Proof of Lemma \ref{lem:intermezzo}]
Here, we develop an argument which is contained in \cite[Chapter 9]{l:GT}.
By noticing that
$a\cdot\nabla\Delta u = \divrg(\Delta u a)-(\divrg a)\Delta u$,
we can rewrite \eqref{eq:41.1} in the form
\begin{equation*}
\sum_{|\alpha|,|\beta|\leq 2}D^\alpha(a_{\alpha\beta}D^\beta u)=0.
\end{equation*}
Let $\sigma\in\left[\frac{1}{2},1\right)$, $\sigma'=\frac{1+\sigma}{2}$ and let us notice that
\begin{equation}
    \label{eq:3a.1}
\sigma'-\sigma = \frac{1-\sigma}{2}, \qquad 1-\sigma = 2(1-\sigma').
\end{equation}
Let $\xi\in C^\infty_0(\R^2)$ be such that
\begin{equation*}
	0\leq\xi\leq 1,
\end{equation*}
\begin{equation*}
	\xi=1, \hbox{ for } |x|\leq \sigma,
\end{equation*}
\begin{equation*}
	\xi=0, \hbox{ for } |x|\geq \sigma',
\end{equation*}
\begin{equation*}
	|D^k(\xi)|\leq \frac{C}{(\sigma'-\sigma)^k}, \hbox{ for } \sigma\leq \sigma', k=0,1,2.
\end{equation*}
By straightforward computations we have that
\begin{equation*}
\sum_{|\alpha|,|\beta|\leq 2}D^\alpha(a_{\alpha\beta}D^\beta (u\xi))=f,
\end{equation*}
with
\begin{equation*}
f=\sum_{|\alpha|,|\beta|\leq 2}\sum_ {\overset{\scriptstyle \delta_2\leq\alpha}{\scriptstyle
\delta_2\neq0}}{\alpha \choose \beta}
D^{\alpha-\delta_2}a_{\alpha\beta}D^\beta u)D^{\delta_2}\xi+
\sum_{|\alpha|,|\beta|\leq 2}D^\alpha\left[a_{\alpha\beta}
\sum_ {\overset{\scriptstyle \delta_1\leq\beta}{\scriptstyle
\delta_1\neq0}}{\beta \choose \delta_1}
D^{\beta-\delta_1}uD^{\delta_1}\xi\right].
\end{equation*}
By standard regularity estimates (see for instance \cite[Theorem 9.8]{l:a65},
\begin{equation}
    \label{eq:8a.1}
\|u\xi\|_{H^{4+k}(B_1^+)}\leq C\left(\|u\xi\|_{L^{2}(B_1^+)}+
\|f\|_{H^{k}(B_1^+)}\right).
\end{equation}
On the other hand, it follows trivially that
\begin{equation}
    \label{eq:8a.2}
\|f\|_{H^{k}(B_1^+)}\leq CM_1 \sum_{h=0}^{3+k}\frac{1}{(1-\sigma')^{4+k-h}}\|D^h u\|_{L^{2}(B_{\sigma'}^+)}.
\end{equation}
By inserting \eqref{eq:8a.2} in \eqref{eq:8a.1}, by multiplying both members by
$(1-\sigma')^{4+k}$ and by recalling \eqref{eq:3a.1}, we have
\begin{equation}
    \label{eq:8a.3}
(1-\sigma)^{4+k}\|D^{4+k}u\|_{L^{2}(B_{\sigma}^+)}\leq C
\left(\|u\|_{L^{2}(B_1^+)}+\sum_{h=1}^{3+k}(1-\sigma')^h
\|D^{h}u\|_{L^{2}(B_{\sigma'}^+)}
\right)
\end{equation}
Setting
\begin{equation*}
\Phi_j=\sup_{\sigma\in\left[\frac{1}{2},1\right)}(1-\sigma)^j
\|D^{j}u\|_{L^{2}(B_{\sigma}^+)},
\end{equation*}
{}from \eqref{eq:8a.3} we obtain
\begin{equation}
    \label{eq:9a.2}
\Phi_{4+k}\leq C\left(A_{2+k}+
\Phi_{3+k}\right).
\end{equation}
where
\begin{equation*}
A_{2+k}=\|u\|_{L^{2}(B_1^+)}+
\sum_{h=1}^{2+k}\Phi_h.
\end{equation*}
By the interpolation estimate \eqref{eq:3a.2} we have that, for every $\epsilon$, $0<\epsilon<1$ and for every $h\in \N$, $1\leq h\leq 3+k$,
\begin{equation}
    \label{eq:9a.3}
\|D^{h}u\|_{L^{2}(B_{\sigma}^+)}\leq C\left(
\epsilon\|D^{4+k}u\|_{L^{2}(B_{\sigma}^+)}+
\epsilon^{-\frac{h}{4+k-h}}\|u\|_{L^{2}(B_{\sigma}^+)}\right).
\end{equation}
Let $\gamma>0$ and let $\sigma_\gamma\in \left[\frac{1}{2},1\right)$ such that
\begin{equation}
    \label{eq:9a.4}
		\Phi_{3+k}\leq(1-\sigma_\gamma)^{3+k}
\|D^{3+k}u\|_{L^{2}(B_{\sigma_\gamma}^+)}+\gamma.
\end{equation}
By applying \eqref{eq:9a.3} with $h=3+k$, $\epsilon=(1-\sigma_\gamma)\widetilde{\epsilon}$, $\sigma = \sigma_\gamma$, we have
\begin{equation*}
(1-\sigma_\gamma)^{3+k}\|D^{3+k}u\|_{L^{2}(B_{\sigma_\gamma}^+)}\leq
\left(
\widetilde{\epsilon}(1-\sigma_\gamma)^{4+k}\|D^{4+k}u\|_{L^{2}(B_{\sigma_\gamma}^+)}+
\widetilde{\epsilon}^{-(3+k)}\|u\|_{L^{2}(B_{\sigma_\gamma}^+)}\right),
\end{equation*}
so that, by \eqref{eq:9a.4} and by the arbitrariness of $\gamma$, we have
\begin{equation*}
\Phi_{3+k}\leq C
\left(
\widetilde{\epsilon}\Phi_{4+k}+
\widetilde{\epsilon}^{-(3+k)}\|u\|_{L^{2}(B_{1}^+)}\right).
\end{equation*}
By inserting this last inequality in \eqref{eq:9a.2}, we get
\begin{equation*}
\Phi_{4+k}\leq C
\left(A_{2+k}+
\widetilde{\epsilon}^{-(3+k)}\|u\|_{L^{2}(B_{1}^+)}+
\widetilde{\epsilon}\Phi_{4+k}\right),
\end{equation*}
which gives, for $\epsilon =\frac{1}{2C+1}$,
\begin{equation*}
\Phi_{4+k}\leq C
\left(\|u\|_{L^{2}(B_{1}^+)}+
\sum_{h=1}^{2+k}\Phi_{h}\right).
\end{equation*}
By proceeding similarly, we get
\begin{equation*}
\Phi_{4+k}\leq C
\|u\|_{L^{2}(B_{1}^+)},
\end{equation*}
so that
\begin{equation}
    \label{eq:12a.1}
	\|D^{4+k}u\|_{L^{2}(B_{\frac{1}{2}}^+)}	
		\leq2^{4+k}C\|u\|_{L^{2}(B_{1}^+)}, \qquad k=0,1,2.
\end{equation}
\end{proof}
By applying \eqref{eq:9a.3} for a fixed $\epsilon$, $\sigma=\frac{1}{2}$, we can estimates
the derivatives of order $h$, $1\leq h\leq 3$,
\begin{equation}
   \label{eq:12a.1bis}
	\|D^{h}u\|_{L^{2}(B_{\frac{1}{2}}^+)}	
		\leq C\left(\|D^{4+k}u\|_{L^{2}(B_{\frac{1}{2}}^+)}+
		\|u\|_{L^{2}(B_{\frac{1}{2}}^+)}\right).
\end{equation}
By \eqref{eq:12a.1}, \eqref{eq:12a.1bis}, we have
\begin{equation*}
	\|D^{h}u\|_{L^{2}(B_{\frac{1}{2}}^+)}	
		\leq C\|u\|_{L^{2}(B_{1}^+)}, \qquad \hbox{ for } h=1,...,6.
\end{equation*}
By employing an homothety, we obtain \eqref{eq:12a.2}.

\medskip
\noindent
\emph{Acknowledgement:} The authors wish to thank Antonino Morassi for fruitful discussions on the subject of this work.

\bibliographystyle{plain}

\end{document}